\documentclass[10pt]{amsart}
\usepackage{amssymb,caption,mathrsfs,pgf,stmaryrd,tikz}
\usepackage[unicode=true, pdfborder={0 0 0}, bookmarksdepth=-1]{hyperref}
\usetikzlibrary{decorations.pathreplacing}

\usepackage[all]{xy}

\makeatletter
  \def\@wrindex#1{%
    \protected@write\@indexfile{}%
      {\string\indexentry{#1}{ \S\thesubsection (p.\thepage)}}
    \endgroup
  \@esphack}

\makeatother
\makeindex

\newcommand{\sha}{\shuffle}

\usepackage{rotating}

\usepackage{shuffle}

\usepackage{amscd}
\usepackage{epsfig}
\author{Benjamin Enriquez}
\address{IRMA (UMR 7501) et D\'epartement de Math\'ematiques, Universit\'e de Strasbourg, 7 rue Ren\'e-Descartes, 67084 Strasbourg (France)}
\email{b.enriquez@math.unistra.fr}
\author{Federico Zerbini}
\address{University of Oxford, Mathematical Institute, Andrew Wiles Building, Radcliffe Observatory Quarter (550), Woodstock Road, Oxford, OX2 6GG (UK)}
\email{federico.zerbini@maths.ox.ac.uk}
\newtheorem{thm}{Theorem}[section]
\newtheorem{lem}[thm]{Lemma}
\newtheorem{lemdef}[thm]{Lemma-Definition}
\newtheorem{cor}[thm]{Corollary}
\newtheorem{prop}[thm]{Proposition}

{\theoremstyle{definition} \newtheorem{rem}[thm]{Remark}}
{\theoremstyle{definition} \newtheorem{defn}[thm]{Definition}}

{\theoremstyle{remark} }

\numberwithin{equation}{subsection}

\DeclareMathOperator{\GL}{\Gamma}

\newcommand{\GLarg}[3]{\tilde{\GL}\left(\begin{smallmatrix}#1\\#2\end{smallmatrix};#3\right)}

\begin{document}
\baselineskip 16pt 

\title[Elliptic hyperlogarithms]{Elliptic hyperlogarithms}

\begin{abstract}
Let $\mathcal E$ be a complex elliptic curve and $S$ be a non-empty finite subset of $\mathcal E$. We show that the functions 
$\tilde\Gamma$ introduced in \cite{BDDT} out of string theory motivations give rise to a basis of the minimal algebra 
$A_{\mathcal E\smallsetminus S}$ of holomorphic multivalued functions on $\mathcal E\smallsetminus S$ which is stable under 
integration, introduced in \cite{EZ2}; this basis is alternative to the basis of $A_{\mathcal E\smallsetminus S}$ constructed 
in {\it loc. cit.} using elliptic analogues of the hyperlogarithm functions.
\end{abstract}
\maketitle

{\small \setcounter{tocdepth}{2}
\tableofcontents}

\section{Introduction}

The hyperlogarithms (HLs) are a family of multivalued holomorphic functions on the punctured complex plane, 
which were first introduced in \cite{P,LD}. They are natural one-variable specialisations of the family of 
multiple polylogarithms, which have found numerous applications in mathematics (see \cite{surv:mpl} for 
a survey, as well as \cite{Br}). 
The HLs also found applications in particle physics (see \cite{Pan} and references therein).  
If $S\subset\mathbb C$ is a finite subset, the HLs on $\mathbb C\smallsetminus S$ are defined 
by iterated integration of a flat connection on a trivial bundle over $\mathbb C\smallsetminus S$ with values in a 
suitable free Lie algebra (\cite{Br}, \S5.1), which corresponds to a Maurer-Cartan element on $\mathbb C\smallsetminus S$ with 
values in this Lie algebra, and which is closely related with the KZ connection (see \cite{EZ2}, Rem. 2.20). 
The product of the algebra $\mathcal H_{\mathbb C\smallsetminus S}$ generated by the HLs on $\mathbb C\smallsetminus S$
with the algebra $\mathcal O(\mathbb C\smallsetminus S)$ of regular functions on $\mathbb C\smallsetminus S$ is an algebra of 
 multivalued holomorphic functions on $\mathbb C\smallsetminus S$, which is stable under the operators 
 $\mathrm{int}_\omega : f\mapsto (z\mapsto \int_{z_0}^z f\omega)$ for $\omega$ in the space 
 $\Omega(\mathbb C\smallsetminus S)$ of regular differentials on $\mathbb C\smallsetminus S$, where~$z_0$ is a fixed point of $\mathbb C\smallsetminus S$ (\cite{Br}, Cor.~5.6), 
 and it coincides with the minimal algebra with these properties $A_{\mathbb C\smallsetminus S}$ 
 (see \cite{EZ2}, Thm.~A). It is proven in \cite{Br}, \S5.2, that the product of $\mathcal H_{\mathbb C\smallsetminus S}$ 
 with $\mathcal O(\mathbb C\smallsetminus S)$ is isomorphic to $\mathcal O(\mathbb C\smallsetminus S)\otimes \mathcal H_{\mathbb C\smallsetminus S}$, and that $\mathcal H_{\mathbb C\smallsetminus S}$ is isomorphic to an explicit shuffle algebra; this enables to prove that this product is a free module over $\mathcal O(\mathbb C\smallsetminus S)$, 
and that a basis is given by the HLs (see \cite{Br} Cor. 5.6, and also \cite{DDMS}).

When $\mathbb C\smallsetminus S$  is replaced by an arbitrary affine curve~$C$, one can similarly attach to any 
Maurer-Cartan element~$J$, non-degenerate in the sense of \cite{EZ2}, an algebra $\mathcal H_C(J)$ of multivalued 
holomorphic functions on~$C$ (see \cite{EZ2}, \S2.3). The above result for the punctured complex plane extends to this case, 
namely the product of the algebra $\mathcal H_C(J)$ with the algebra $\mathcal O(C)$ of regular functions on $C$ 
coincides with the minimal subalgebra~$A_C$ of the algebra of multivalued functions on~$C$ that is stable under the 
analogues $\mathrm{int}_\omega$ of the above operators, where $\omega$ runs over the space $\Omega(C)$ of regular 
differentials on $C$ (\cite{EZ2}, Thm. A(b)), and is therefore independent of~$J$. The algebra~$A_C$ can also be 
characterised as the set of multivalued holomorphic functions on~$C$ which have moderate growth at the cusps of~$C$ and have 
unipotent monodromy, 
in the sense that the representation of the fundamental group of~$C$ generated by such a function is a finite iterated extension of the
trivial representation (\cite{EZ2}, Thm.~C). 
In \cite{EZ2}, Thm. A(b), one shows the isomorphisms of $A_C$ with $\mathcal O(C) \otimes \mathcal H_C(J)$ and of 
$\mathcal H_C(J)$ with an explicit shuffle algebra; this enables us to attach to a family of $\Omega(C)$, whose image is 
a basis of $\Omega(C)/d\mathcal O(C)$, a basis of $A_C$ as an $\mathcal O(C)$-module (Def. \ref{def:hyperlog} and 
Lems.~\ref{lem:basis:2906} and \ref{new:lem:1203}), whose elements are analogues of the classical HLs. 

In this article, we will be interested in the particular case where $C$ has genus one; more precisely, 
we take it to coincide with the affine curve $\mathcal E_S:=\mathcal E\smallsetminus S$, 
for $\mathcal E$ a complex elliptic curve and $S\subset \mathcal E$ a finite non-empty subset. 
In §\ref{NEWSECTION}, we exhibit an explicit family 
of $\Omega(\mathcal E_S)$ whose image is a basis of $\Omega(\mathcal E_S)/d\mathcal O(\mathcal E_S)$; 
according to Lems. \ref{lem:basis:2906} and 
\ref{new:lem:1203}, this gives rise to a basis of $A_{\mathcal E_S}$ as an $\mathcal O(\mathcal E_S)$-module, whose elements will 
be called \emph{elliptic hyperlogarithms}.  




Recent development in particle physics led to the introduction of two (closely related) classes of 
multivalued functions $\tilde \Gamma$ and $\mathrm E_3$ on $\mathcal E_S$ 
(see \cite{BDDT}, and \cite{BMMS,BK} for applications in string theory). The functions $\tilde \Gamma$
are defined as iterated integrals of multivalued holomorphic differential forms on~$\mathbb C\smallsetminus pr^{-1}(S)$, 
$pr : \mathbb C\to\mathcal E$ being a universal covering map. One can show that the functions $\tilde \Gamma$ 
may also be defined, using iterated integration, 
out of a holomorphic flat connection on a non-trivial bundle over $\mathcal E_S$, which
can be identified with a connection obtained upon restriction from the ``universal KZB'' 
connection of \cite{CEE} (or \cite{LR} if $|S|=1$). 

The main result of this article is a proof that the collection of functions $\tilde\Gamma$ gives rise to a basis of $A_{\mathcal E_S}$ 
over $\mathcal O(\mathcal E_S)$ (see Thm. \ref{THM:MAIN}), alternative to the family of elliptic HLs from §\ref{NEWSECTION}. 
This relies on two main steps: 

\begin{itemize}
    \item[(a)] We prove that $A_{\mathcal E_S}$ is equal to an algebra $\mathcal G$ generated by the functions 
    $\tilde\Gamma$ (Thm. \ref{THMEQ}): the inclusion $\mathcal G \subset A_{\mathcal E_S}$ (Cor. \ref{cor:1704}) is based 
    on the characterization of $A_{\mathcal E_S}$ as the set of multivalued functions sharing certain differential 
    properties, and the inclusion $\mathcal G \supset A_{\mathcal E_S}$ is based on the characterization of 
    $A_{\mathcal E_S}$  as the minimal algebra of multivalued functions which are stable under the endomorphisms $\mathrm{int}_\omega$
    (see Thm. \ref{THMEQ}).  

    \item[(b)] We prove a linear independence result for the functions $\tilde\Gamma$ (see Prop. \ref{thm:main:bis}), 
    based on a criterion of \cite{DDMS} and on a precise analysis of differential algebras attached to $\mathcal E_S$. 
\end{itemize}


The definition of the functions $\mathrm E_3$ relies on the additional datum of a degree 2 covering map $\pi :\mathcal  E\to
\mathbb P^1_{\mathbb C}$ and a finite subset $S_0\subset \mathbb P^1_{\mathbb C}$. It is shown in \cite{BDDT}, \S5 (see also Prop. 
\ref{prop:iso:spaces:0307} in this paper) that 
the vector spaces of multivalued functions generated by the $\tilde\Gamma$ and by the~$\mathrm E_3$ coincide if $S=\pi^{-1}(S_0)$. 
In \cite{BDDT} it also shown that an algebra constructed out of the functions $\mathrm E_3$ is stable under integration 
(\S6, based on explicit computation). A consequence of~(a) is an alternative proof of this 
stability result (see Prop. \ref{prop:bddt:2704}).   

\subsection{History and outlook} The first appearance of elliptic analogues of hyperlogarithms dates back to \cite{BL}, \S10.1, 
where the authors construct functions called ``elliptic Debye hyperlogarithm'' via averaging of classical polylogarithm functions, 
generalising analogous constructions of \cite{Bl,Zag1,Lev}. These functions are holomorphic and multivalued, and we expect them 
to be contained in the algebra $A_{\mathcal E_S}$ (work in progress). They are closely related (see \cite{BL}, \S10.3) to certain 
real-analytic 
functions defined via iterated integration out of a real-analytic version of the KZB connection of \cite{CEE}; such functions may be 
seen as real-analytic analogues of elliptic hyperlogarithms, and were recently generalised to higher-genus Riemann surfaces in 
\cite{DHS}. 

\subsection{Organisation of the article} 

In \S\ref{sect:2:0407} we recall from \cite{EZ2} the construction and properties of the algebra~$A_C$ attached to a curve~$C$,
and we specialise these results to the situation of an elliptic curve; in particular, we exhibit a basis of 
$A_{\mathcal E_S}$ over $\mathcal O(\mathcal E_S)$ which consists of elliptic HLs (§\ref{NEWSECTION}). 
We then recall in \S\ref{sect:tilde:Gamma:0407} the definition  of the functions $\tilde\Gamma$ (\cite{BDDT}), 
paying special attention to regularisation issues. \S\ref{sect:5:2606} is devoted to proving our main result: the family of functions $\tilde\Gamma$ 
gives rise to a basis of $A_{\mathcal E_S}$ over $\mathcal O(\mathcal E_S)$ (Thm. \ref{THM:MAIN}), alternative to the basis from 
\S\ref{NEWSECTION}.
In \S\ref{sect:6:0407}, we apply results from \S\ref{sect:5:2606} to deduce a new proof of a result of \cite{BDDT} on the stability 
under integration of an algebra~$\mathcal A_3$ constructed out of the functions $\mathrm E_3$.

\subsection{Conventions, notation}

In this paper, all vector spaces and algebras will be understood to be with base field $\mathbb C$.

\subsubsection{Filtrations}
A {\it filtered vector space}
$F_\bullet V$ is the data of a vector space $V$ and a collection of subspaces $(F_nV)_{n\geq 0}$ 
such that $F_nV\subset F_{n+1}V$ for any $n\geq 0$; by convention, $F_{-1}V=0$. A morphism of filtered vector spaces 
$F_\bullet V\to F_\bullet W$ is a linear map $f : V\to W$, such that $f(F_nV)\subset F_nW$ for any $n\geq0$.  
The {\it associated graded} of the filtered vector space $F_\bullet V$ is the graded vector space
$\mathrm{gr}(V):=\oplus_{n\geq0}\mathrm{gr}_n(V)$, where $\mathrm{gr}_n(V):=F_nV/F_{n-1}V$. 
If $F_\bullet V$ is a filtration of a vector space $V$, 
we denote by $x\mapsto\overline x$ the projection map 
$F_nV\to\mathrm{gr}_n(V)$ for any $n\geq 0$. If $F_\bullet V$ and $F_\bullet W$ are filtered vector spaces, 
their tensor product $F_\bullet(V\otimes W)$ is the data of $V\otimes W$, equipped with the filtration given 
by $F_n(V\otimes W):=\sum_{p=0}^n F_pV\otimes F_{n-p}W$. A {\it filtered algebra} is an algebra $A$ equipped with a
filtration $F_\bullet A$, such that the product of $A$ is a morphism of filtered vector spaces $F_\bullet(A\otimes A)
\to F_\bullet A$. 

The {\it total space} of a filtered vector space $F_\bullet V$ is the subspace $F_\infty V:=\cup_{n\geq 0}F_nV$ of $V$. 
The total space $F_\infty A$ of a filtered algebra $F_\bullet A$ is then a subalgebra of $A$. 

\subsubsection{Meromorphic functions}\label{merom:funs}

For $\Sigma\subset \mathbb C$ a discrete subset, 
let $\mathcal O_{mer}(\mathbb C,\Sigma)$ be the algebra 
of meromorphic functions on $\mathbb C$ with sets of poles contained in~$\Sigma$. It is contained in the algebra $\mathcal O_{hol}(\mathbb C\smallsetminus\Sigma)$ of holomorphic functions on $\mathbb C\smallsetminus\Sigma$, and it is equipped with a derivation $\partial:=\partial/\partial z$; we use the notation $f':=\partial (f)$. For $a\in\mathbb C$, 
we denote by $T_a$ the automorphism of the algebra of meromorphic functions on $\mathbb C$ (equal to $\cup_\Sigma
\mathcal O_{mer}(\mathbb C,\Sigma)$) defined by $T_af:=(z\mapsto f(z-a))$.

\section{Minimal stable subalgebras and elliptic hyperlogarithms (based on~\cite{EZ2})}\label{sect:2:0407}

In §§\ref{sect:21:2704}-\ref{basis:A_C}, we recall material from \cite{EZ2}. More precisely, to each smooth affine complex curve $C$, we 
attach a 
subalgebra $A_C$ of the algebra of holomorphic functions $\mathcal O_{hol}(\tilde C)$ on a universal cover of~$C$ (see Lem.-Def. 
\ref{lem:ez2:1704}), called the minimal stable subalgebra (\S\ref{sect:21:2704}); we define an algebra filtration 
$F_\bullet^\delta\mathcal O_{hol}(\tilde C)$ of $\mathcal O_{hol}(\tilde C)$ with 
$A_C=F_\infty^\delta\mathcal O_{hol}(\tilde C)$ (see Thm.~\ref{thm:2:3:2606}(b)) (see \S\ref{sect:2:2:2606}); 
we express an $\mathcal O(C)$-basis of $A_C$ in terms of HL functions associated with $C$ (\S\ref{basis:A_C}).
In §\ref{sect:setup}, we introduce the elliptic setup needed for the main results of the paper; then $C=\mathcal E_S$. 
In §\ref{NEWSECTION}, we formulate a result of \cite{EZ2} in the elliptic context, namely the expression of a
$\mathcal O(\mathcal E_S)$-basis of $A_{\mathcal E_S}$ in terms of elliptic HL functions. 


\subsection{The minimal stable subalgebra $A_C$ associated with a curve $C$}\label{sect:21:2704}

Let $C$ be a smooth affine complex curve, which we view as a Riemann surface. Let us fix a universal cover $p:\tilde C\to C$ and let 
$\mathcal O_{hol}(\tilde C)$ be the algebra of holomorphic functions on $\tilde C$. 

Let $\Omega(C)$ be the space of regular differentials on $C$. 
For any $\omega\in\Omega(C)$, the map $f\mapsto f\cdot p^*(\omega)$ is a linear map 
$\mathcal O_{hol}(\tilde C)\to\Omega_{hol}(\tilde C)$, where  
$\Omega_{hol}(V)$ is the space of holomorphic differential 1-forms on a complex manifold $V$. 
For any $z_0\in \tilde C$ and any $\alpha\in \Omega_{hol}(\tilde C)$, the assignment
$\alpha\mapsto[z\mapsto \int_{z_0}^z \alpha]$ is well-defined as~$\tilde C$ is simply-connected, and defines a linear map 
$\Omega_{hol}(\tilde C)\to\mathcal O_{hol}(\tilde C)$. 
For $\omega\in\Omega(C)$, let $\mathrm{int}_{\omega}$ be the linear endomorphism of $\mathcal O_{hol}(\tilde C)$ given by $f\mapsto 
[z\mapsto \int_{z_0}^z f\cdot p^*(\omega)]$. 

\begin{defn}\label{def:stable:subalg}
 A {\it stable subalgebra} of $\mathcal O_{hol}(\tilde C)$ is a subalgebra with unit, which is stable under the endomorphism 
 $\mathrm{int}_\omega$ for any $\omega\in\Omega(C)$. 
\end{defn}

Although the definition of the operators $\mathrm{int}_{\omega}$ depends on a choice of $z_0$, one checks that the notion of 
stable subalgebra is independent of such a choice. 

\begin{lemdef}[see \cite{EZ2}, \S5.3]\label{lem:ez2:1704}
  If $A_C:=\cap_{A\mathrm{\ stable\ subalgebra\ of\ }\mathcal O_{hol}(\tilde C)}A$, then $A_C$ is a stable subalgebra of $\mathcal O_{hol}(\tilde C)$, which
  is minimal for the inclusion; it is called the {\it minimal stable subalgebra} of $\mathcal O_{hol}(\tilde C)$.  
\end{lemdef}

\subsection{Relation of $A_C$ with the differential filtration}\label{sect:2:2:2606}

Define inductively a collection of subspaces $(F^\delta_n\mathcal O_{hol}(\tilde C))_{n\geq 0}$ of $\mathcal O_{hol}(\tilde C)$ by 
$F^\delta_0\mathcal O_{hol}(\tilde C):=\mathbb C$, and $F^\delta_{n+1}\mathcal O_{hol}(\tilde C):=\{f\in \mathcal O_{hol}(\tilde C)\,|\,df\in F^\delta_n\mathcal O_{hol}(\tilde C)\cdot p^*\Omega(C)\}$ for $n\geq0$. Define also the subspace 
$F^\delta_\infty\mathcal O_{hol}(\tilde C):=\cup_{n\geq0}F^\delta_n\mathcal O_{hol}(\tilde C)$ of $\mathcal O_{hol}(\tilde C)$. 

Let us denote by $\mathcal O(C)$ the algebra of regular functions on $C$. 

\begin{thm}[see \cite{EZ2}] \label{thm:2:3:2606}
(a)   $F^\delta_\bullet\mathcal O_{hol}(\tilde C)$ is an 
algebra filtration of $\mathcal O_{hol}(\tilde C)$, such that 
$p^*\mathcal O(C)\subset F^\delta_1\mathcal O_{hol}(\tilde C)$. 

(b) $F^\delta_\infty\mathcal O_{hol}(\tilde C)=A_C$. 
\end{thm}

\begin{proof}
The fact that $F^\delta_\bullet$ is an algebra filtration is proven in \cite{EZ2}, Prop. 5.5(a). The inclusion $p^*\mathcal O(C)\subset F^\delta_1\mathcal O_{hol}(\tilde C)$ follows from the $n=0$ case of Prop. 5.5(b) in \cite{EZ2}. Statement (b) follows from \cite{EZ2}, Thm.~C. 
\end{proof}

\subsection{Analogues of HLs on a curve $C$}\label{basis:A_C}

For $W$ a complex vector space, denote by $\mathrm{Sh}(W)$ the shuffle Hopf algebra of $W$. This is 
the vector space $\oplus_{k\geq0}W^{\otimes k}$ (the element $w_1\otimes\cdots\otimes w_k\in W^{\otimes k}$ is 
denoted $[w_1|\cdots|w_k]$), equipped with the shuffle product $\shuffle$, which is commutative, the deconcatenation coproduct 
$\Delta$ given by $[w_1|\cdots|w_k]\mapsto\sum_{j=0}^k[w_1|\ldots|w_j]\otimes[w_{j+1}|\ldots|w_k]$, 
the counit $\varepsilon$ given by $1\mapsto 1$ and $[w_1|\cdots|w_k]\mapsto 0$ for $k>0$, 
and the antipode given by $[w_1|\ldots|w_k]\mapsto (-1)^k [w_k|\ldots|w_1]$. A linear map $\phi : W\to W'$
induces a Hopf algebra morphism $\mathrm{Sh}(W)\to\mathrm{Sh}(W')$, also denoted $\phi$. 

Let $z_0\in\tilde C$. There is a unique linear map 
\begin{equation}\label{def:I:x0}
I_{z_0} : \mathrm{Sh}(\Omega(C))\to\mathcal O_{hol}(\tilde C)    
\end{equation}
satisfying the identities $I_{z_0}(a)(z_0)=\varepsilon(a)$ and $d(I_{z_0}([\omega_1|\cdots|\omega_n])=
I_{z_0}([\omega_1|\cdots|\omega_{n-1}])\cdot \omega_n$ for any $n\geq0$ and $\omega_1,\ldots,\omega_n\in\Omega(C)$, 
so that $I_{z_0}([\omega_1|\cdots |\omega_n])=\mathrm{int}_{\omega_n}(I_{z_0}([\omega_1|\cdots |\omega_{n-1}]))$;
such a map~$I_{z_0}$ is an algebra morphism, given by iterated integration based at $z_0$ (see \cite{EZ2}, Lem. 2.2). 

Set $\mathrm H_C:=\Omega(C)/d\mathcal O(C)$. Let $\sigma : \mathrm H_C\to \Omega(C)$ be a section of the projection map 
$\Omega(C)\to\mathrm H_C$.

\begin{lem}[\cite{EZ2}]\label{lem:06031758}
(a) Let $\sigma_*:\mathrm{Sh}(\mathrm H_C)\to \mathrm{Sh}(\Omega(C))$ be algebra morphism induced by~$\sigma$, then $\tilde f_{\sigma,z_0}:=I_{z_0}\circ\sigma : \mathrm{Sh}(\mathrm H_C)\to \mathcal O_{hol}(\tilde C)$ is an injective algebra morphism. 

(b) The image of $\tilde f_{\sigma,z_0}$ is independent of $z_0$; it is a subalgebra of $\mathcal O_{hol}(\tilde C)$, denoted $\mathcal H_C(\sigma)$.

(c) The composition $\mathcal O(C) \otimes \mathcal H_C(\sigma)\hookrightarrow  \mathcal O_{hol}(\tilde C)\otimes  
\mathcal O_{hol}(\tilde C)\to \mathcal O_{hol}(\tilde C)$ of the canonical injections with the product map of 
$\mathcal O_{hol}(\tilde C)$ yields an algebra isomorphism  
$\mathcal O(C)\otimes \mathcal H_C(\sigma)\simeq A_C$. 
\end{lem}

\begin{proof}
Statements (a) and (c) are consequences of Thm. A(b) in \cite{EZ2} with $J$ equal to the element 
$J_\sigma$ defined in Def. 1.2 of \cite{EZ2}. Statement (b) follows from \cite{EZ2}, Thm. A(a).
\end{proof}

\begin{lem}\label{lem:basis:2906}
Let $d:=\mathrm{dim}(\mathrm H_C)$ and $(h_i)_{i\in[\![1,d]\!]}$ be\footnote{We set $[\![1,d]\!]:=\{1,\ldots,d\}$.} 
a $\mathbb C$-basis of $\mathrm H_C$. 
Then the family
    $$
(\tilde f_{\sigma,z_0}([h_{i_1}|\ldots|h_{i_s}]))_{s\geq0,(i_1,\ldots,i_s)\in[\![1,d]\!]^s}
    $$
is a $\mathbb C$-basis of $\mathcal H_C(\sigma)$, as well as an $\mathcal O(C)$-basis of $A_C$. 
\end{lem}

\begin{proof}
    The first statement follows from \cite{EZ2}, Prop. 5.10, and the second statement follows from its combination 
    with Lem. \ref{lem:06031758}(c).  
\end{proof}

\begin{defn}\label{def:hyperlog}
If $(\alpha_i)_{i\in[\![1,d]\!]}$ is a family of $\Omega(C)$ whose image in $\mathrm H_C=\Omega(C)/d(\mathcal O(C))$ 
is a $\mathbb C$-basis, set 
\begin{equation}\label{def:hyperlog:fun}
L_{\alpha_{i_1},\ldots,\alpha_{i_s}}:=I_{z_0}([\alpha_{i_1}|\ldots|\alpha_{i_s}])
\end{equation}
for any $s\geq 0$ and $i_1,\ldots,i_s \in [\![1,d]\!]$. 
\end{defn}

We call \eqref{def:hyperlog:fun} the {\it hyperlogarithm functions associated with the family $(\alpha_i)_{i\in[\![1,d]\!]}$.} 

\begin{lem}\label{new:lem:1203}
If $(\alpha_i)_{i\in[\![1,d]\!]}$ is a family of $\Omega(C)$ as in Def. \ref{def:hyperlog}, then there exists a
unique pair $(\sigma,(h_i)_i)$ where $\sigma : \mathrm H_C\to \Omega(C)$ is a section of the projection 
$\Omega(C)\to\mathrm H_C$ and $(h_i)_{i\in[\![1,d]\!]}$ is a $\mathbb C$-basis of $\mathrm H_C$, such that 
$\alpha_i=\sigma(h_i)$ for any $i$. One then has the identity 
\begin{equation}\label{ID:L:TILDEF}
L_{\alpha_{i_1},\ldots,\alpha_{i_s}}=\tilde f_{\sigma,z_0}([h_{i_1}|\ldots|h_{i_s}]). 
\end{equation}
\end{lem}

\begin{proof} This is straightforward. 
\end{proof}



\subsection{Background on elliptic curves}\label{sect:setup}

\subsubsection{Elliptic curves and coverings}\label{sect:3:1:2606}

Let us denote by $\mathfrak H$ the complex upper half-plane. Until the end of \S\ref{sect:5:2606}, an element $\tau$ in $\mathfrak H$ is fixed. 
Let $\Lambda:=\mathbb{Z}+\mathbb{Z}\tau\subset \mathbb{C}$. 

\begin{defn}
We set $\mathcal{E}:=\mathbb{C}/\Lambda$, and we denote $pr : \mathbb C\to\mathbb{C}/\Lambda=\mathcal E$
the canonical projection. 
\end{defn} 

$\mathcal E$ is a compact genus-one Riemann surface, and $pr$  is a universal cover map for $\mathcal E$.    

\begin{defn}
Let $S$ be a finite subset of $\mathcal E$ containing $pr(0)$. 
We set $\mathcal E_S:=\mathcal{E}\smallsetminus S$, and we fix a universal cover $p:\tilde{\mathcal E}_S\to \mathcal E_S$.    
\end{defn}
 
The restriction of $pr$ is a covering map $\mathbb C\smallsetminus pr^{-1}(S)\to \mathcal E_S$.
The universal cover map $p$ then admits a factorisation $p=\varpi\circ pr$, where 
$\varpi:\tilde{\mathcal E}_S\to \mathbb C\smallsetminus pr^{-1}(S)$ is a universal cover map, so that 
the following diagram commutes
$$
\xymatrix{\tilde{\mathcal E}_S\ar^{\varpi}[rr]\ar_p[rd]&&\mathbb C\smallsetminus pr^{-1}(S)\ar^{pr}[ld]\\&
\mathcal E_S&}
$$
\begin{defn}\label{def:tildeS:3006}
(a) Let $s\mapsto \tilde s$ be a section of the map $pr : pr^{-1}(S)\to S$ such that $pr(0)\mapsto 0$.

(b) We denote by $\tilde S$ the subset of $pr^{-1}(S)$ defined as the image of the map $s\mapsto \tilde s$.    
\end{defn}
    

One has $pr^{-1}(S)=\cup_{s\in S}(\tilde s+\Lambda)$ (equality of subsets of $\mathbb C$). 

\subsubsection{Functions in $\mathcal O_{mer}(\mathbb C,pr^{-1}(S))$}
\label{sect:3:1:2704}


For $r\in\mathbb Z_{\geq 1}$ and $z\in \mathbb C\smallsetminus\Lambda$, let
\begin{equation*}
E_r(z)\,:=\, \sum_{\lambda\in \Lambda}\frac{1}{(z+\lambda)^r}\,\in\mathbb C, 
\end{equation*}
where the series is absolutely convergent for $r>2$, and is defined by Eisenstein summation for $r=1,2$ (see \cite{BMMS}, footnote 8), and 
let $E_r$ be the 
function $z\mapsto E_r(z)$; then $E_r$ belongs to $\mathcal O_{mer}(\mathbb C,\Lambda)$. If $r\geq 2$, then the function 
$E_r$ is elliptic, i.e. $\Lambda$-invariant; on the other hand, 
$E_1$ satisfies the identities\footnote{We set $\mathrm{i}:=\sqrt{-1}$.}
\begin{equation}\label{ellpropE1}
T_{1}E_1\,=\,E_1\,, \quad \quad \quad T_{\tau}E_1\,=\,E_1\,+\,2\pi \mathrm{i}\,.
\end{equation}
For every $r\geq 1$, the function $E_r$ satisfies the identity 
\begin{equation}\label{diffpropE1}
E_r'\,=\,-r\,E_{r+1}\,.
\end{equation}

Moreover, for every $r\geq 1$ consider also the Eisenstein series 
\begin{equation}
e_r\,:=\,\sum_{\lambda \in \Lambda\smallsetminus\{0\}}\frac{1}{\lambda^r}\,\in\mathbb C\,,
\end{equation}
which is again defined via the Eisenstein summation for $r\leq 2$, and vanishes when $r$ is odd.

The Weierstrass $\wp$-function attached to $\Lambda$ is given by $\wp:=E_2-e_2$. One then has 
\begin{equation}\label{relations:wp:Er}
\forall r\geq 0,\quad    \wp^{(r)}=(-1)^r(r+1)!E_{r+2}-e_2\delta_{r,0}. 
\end{equation}

\begin{defn}
We denote by $(g_n)_{n\geq 0}$ the family of meromorphic functions\footnote{Notice the slight change of notation with respect to \cite{BMMS} and \cite{BDDT}, 
 where these functions 
 are denoted~$g^{(n)}$, in order to avoid confusion with $n$-th derivatives.} in 
 $\mathcal O_{mer}(\mathbb C,\Lambda)$ defined via the following identity of generating series
in the formal variable $\alpha$:
\begin{equation}\label{defgn}
\sum_{n\geq 0}g_n\,\alpha^{n-1}\,:=\,\frac{1}{\alpha}\,\exp\Big(-\sum_{r\geq 1}\frac{(-\alpha)^r}{r}\,(E_r-e_r)\Big)\,.
\end{equation}
\end{defn} 

In particular, $g_0=1$ and $g_1=E_1$ has a simple pole with residue~$1$ at all points of $\Lambda$. 
Fourier expansions for the functions $g_n$ can be found in \cite{BMMS}, eqs. (3.35) and (3.36).

\begin{lem}[see \cite{Ma}, Prop. 2.1.3]
The function $g_n$ is regular at $0$ for any $n\geq 2$.     
\end{lem}
 
The left-hand side of \eqref{defgn} is an element of the ring $\mathcal O_{mer}(\mathbb C,\Lambda)((\alpha))$ 
of formal Laurent series with coefficients in $\mathcal O_{mer}(\mathbb C,\Lambda)$, which we denote 
by $F$. 

\begin{rem}
$F$ is the expansion at $\alpha=0$ of a meromorphic function of two variables, known as the Kronecker 
function (\cite{Wei,Zag}).  
\end{rem}

\begin{lemdef}
(a)  For every $a\in pr^{-1}(S)$ and every $n\geq 0$, the function $T_ag_n=(z\mapsto g_n(z-a))$ belongs to 
$\mathcal O_{mer}(\mathbb C,pr^{-1}(S))$.

(b) Set $\omega_{n,a}:=T_a(g_n)\cdot dz$, then $\omega_{n,a}$ belongs to $\Omega_{hol}(\mathbb C
\smallsetminus pr^{-1}(S))$.   
\end{lemdef} 

\begin{proof}
The first statement follows from $T_a(\mathcal O_{mer}(\mathbb C,\Lambda))=\mathcal O_{mer}(\mathbb C,a+\Lambda)$ and $a+\Lambda\subset pr^{-1}(S)$.  The second statement follows from the module structure of 
$\Omega_{hol}(\mathbb C\smallsetminus pr^{-1}(S))$ over $\mathcal O_{hol}(\mathbb C\smallsetminus pr^{-1}(S))$, and the inclusion $\mathcal O_{mer}(\mathbb C,\Lambda)\subset \mathcal O_{hol}(\mathbb C\smallsetminus pr^{-1}(S))$. 
\end{proof}

\subsubsection{Functional equations}\label{subsect:FE}

\begin{lem}
For every $n\geq 1$, one has\footnote{
Eq. \eqref{ellpropgn} shows that $g_n$ is not regular on the whole of $\mathbb C$, contrary to the statement 
of~\cite{BDDT}, three lines after (4.24).} the following equalities in $\mathcal O_{mer}(\mathbb C,\Lambda)$:
\begin{align}
T_{1}g_{n}\,=\,&g_n\,, \quad\quad \quad (T_{\tau}-id)(g_n)\,=\,\sum_{k=0}^{n-1}\frac{(2\pi \mathrm{i})^{n-k}}{(n-k)!}\,
g_k\,,\label{ellpropgn}\\
&g_n'\,=\,\sum_{k=0}^{n-1}(-1)^{n-k}\,E_{n-k+1}\,g_k\,.\label{diffpropgn}
\end{align}
\end{lem}

\begin{proof}
Applying \eqref{ellpropE1} and \eqref{diffpropE1} to the right-hand side of \eqref{defgn}, we obtain that
\begin{align*}
F(z+1,&\alpha)\,=\,F(z,\alpha)\,, \quad\quad \quad F(z+\tau,\alpha)\,=\,e^{-2\pi i\alpha}\,F(z,\alpha)\,,\\
&\frac{\partial}{\partial z}F(z,\alpha)\,=\,\Big(\sum_{r\geq 1}(-\alpha)^r\,E_{r+1}(z)\Big)\,F(z,\alpha)\,,
\end{align*}
(the first line is (3.13) in \cite{BMMS}). The statement follows by comparing these identities  with the left-hand side of~\eqref{defgn}.  
\end{proof}

\begin{cor}
    For any $n\geq 0$ and $a\in pr^{-1}(S)$, one has 
    \begin{equation}\label{FE:omega}
     \omega_{n,a+1}=\omega_{n,a},   \quad  \omega_{n,a+\tau}=\sum_{k=0}^n 
{(2 \pi \mathrm{i})^{n-k}\over (n-k)!}\omega_{k,a}. 
    \end{equation}
\end{cor}

\begin{proof}
 Follows from \eqref{ellpropgn} and from the invariance of $dz$ under $T_a$, $a\in pr^{-1}(S)$.    
\end{proof}


\subsection{An $\mathcal O(\mathcal E_S)$-basis of $A_{\mathcal E_S}$ arising from elliptic HLs}\label{NEWSECTION}

We introduce the shorthand notation $O_S:=\mathcal O(\mathcal E_S)$; then $O_S$ is a subalgebra of 
$\mathcal O_{mer}(\mathbb C,pr^{-1}(S))$. 

\begin{defn}
$\partial$ is the derivation of $O_S$ arising from the restriction of 
the action of $\partial$ on $\mathcal O_{mer}(\mathbb C,pr^{-1}(S))$ (see \S\ref{merom:funs}). 
\end{defn}

Set $S_\star:=S\sqcup\{\star\}$. Define a map $S_\star \to\Omega(\mathcal E_S)$, $s\mapsto \alpha_s$ by 
$\alpha_s:=(T_s-id)(g_1)\cdot dz$ if $s\in S\smallsetminus\{0\}$, $\alpha_\star:=dz$, $\alpha_0:=E_2\cdot dz$. 

\begin{lem}\label{lem:0307}
The family $(\alpha_s)_{s\in S_\star}$ is a family of $\Omega(\mathcal E_S)$ whose image in $\mathrm H_{\mathcal E_S}
=\Omega(\mathcal E_S)/d(O_S)$ 
is a $\mathbb C$-basis. 
\end{lem}

\begin{proof}

The residue map induces a linear map $O_S/\partial(O_S)\to(\mathbb C^S)_0$, where $(\mathbb C^S)_0\subset \mathbb C^S$
is the set of tuples whose sum is equal to $0$. 

If $f\in O_S$ is a representative of an element of the kernel of $O_S/\partial(O_S)\to(\mathbb C^S)_0$, 
then the residue of $fdz$ at any $s\in S$ is zero, so that the function $z\mapsto \int_{z_0}^z f(u)du$ is single-valued around each 
$s\in S$. It follows that there exist $a,b\in\mathbb C$ such that $z\mapsto \int_{z_0}^z f(u)du-az-bg_1(z)$ belongs to $O_S$, and
therefore $f\equiv a+bg'_1=a-bE_2$ mod $\partial(O_S)$. This implies that the kernel of $O_S/\partial(O_S)\to(\mathbb C^S)_0$ is 
spanned by $1$ and $E_2$. 

The map $O_S/\partial(O_S)\to(\mathbb C^S)_0$ is surjective, because a preimage of $(a_s)_s\in (\mathbb C^S)_0$
is $\sum_{s\in S}a_sT_{\tilde s}(g_1)$. Therefore $O_S/\partial(O_S)$ fits in an exact sequence 
$$
0\to\mathrm{Span}_{\mathbb C}(1,E_2)\to O_S/\partial(O_S)\to (\mathbb C^S)_0\to 0. 
$$
A basis of  $O_S/\partial(O_S)$ is therefore the union of a basis of  
$\mathrm{Span}_{\mathbb C}(1,E_2)$ and of the preimage of a basis of $(\mathbb C^S)_0$. 
Since $(1,E_2)$ is linearly independent and since $((T_{\tilde s}-id)(g_1))_{s\in S\smallsetminus\{0\}}$
is such a preimage, the image of 
\begin{equation}\label{family:OS}
1,E_2,((T_{\tilde s}-id)(g_1))_{s\in S\smallsetminus\{0\}}
\end{equation}
in $O_S/\partial(O_S)$ forms a $\mathbb C$-basis of $O_S/\partial(O_S)$.

The linear isomorphism $-\cdot dz : O_S\to \Omega(\mathcal E_S)$, $f\mapsto f\cdot dz$ is such that 
$(-\cdot dz) \circ \partial=d$ (equality of linear maps 
$O_S\to \Omega(\mathcal E_S)$). This linear isomorphism therefore induces a linear isomorphism 
$-\cdot dz : O_S/\partial(O_S)\to\Omega(\mathcal E_S)/d(O_S)$. The statement follows from this 
and from the fact that the claimed family is the image of 
\eqref{family:OS} by $-\cdot dz$. 
\end{proof}

\begin{cor}\label{cor:2203:1124}
The family      
\begin{equation}\label{basis:HL:g=1}
L_{\alpha_{s_1},\ldots,\alpha_{s_k}}\,=\,I_{z_0}([\alpha_{s_1}|\cdots |\alpha_{s_k}]),\quad k\geq0,\quad 
    (s_1,\ldots,s_k)\in S_\star^k.   
\end{equation}   
is an $O_S$-basis of the $O_S$-module $A_{\mathcal E_S}$. 
\end{cor}

\begin{proof}
This follows from Lems. \ref{lem:basis:2906} and \eqref{ID:L:TILDEF}. 
\end{proof}

\section{The multivalued functions $\tilde\Gamma$ from \cite{BDDT}}\label{sect:tilde:Gamma:0407}

The article \cite{BDDT} introduces a family of functions 
$\GLarg{n_1 &n_2 &\ldots &n_r}{a_1 &a_2 &\ldots &a_r}{-}$, with $r\geq 0$, $n_i\geq 0$ and $a_i\in pr^{-1}(S)$
for\footnote{For $a\leq b \in\mathbb Z$, 
we set $[\![a,b]\!]:=\{a,a+1,\ldots,b\}$.} $i\in[\![1,r]\!]$, which are multivalued functions on $\mathbb C\smallsetminus pr^{-1}(S)$ defined as regularised 
iterated integrals of the forms $(\omega_{n,a})_{n\geq 0,a\in pr^{-1}(S)}$ from~\S\ref{sect:setup}, where the integration path starts at 0. 
It follows from the functional equation \eqref{FE:omega} that 
this vector space is spanned by the collection of the 
same functions, where the parameters satisfy the restricted condition $a_i\in\tilde S$ for any $i\in[\![1,r]\!]$. 
Since one of the possible forms, namely $\omega_{1,0}$, has a pole at this point, regularisation is needed for 
the definition of the functions $\tilde\Gamma$ even in the restricted case. 

The regularisation procedure indicated in 
\cite{BDDT} relies on replacing the endpoint 0 by a small~$\epsilon$ and extracting a value from the asymptotic expansion in 
$\epsilon$; this procedure was rigourously 
carried out in a similar situation in \cite{BGF}, \S3.8, which is based on the ``tangential base points'' ideas of \cite{De}, 
\S15. On the other hand, the regularisation procedure indicated in \cite{BK} 
for the definition of the functions $\GLarg{n_1 &n_2 &\ldots &n_r}{a_1 &a_2 &\ldots &a_r}{-}$ is an analogue of 
the procedure from~\cite{Br}, \S5.1 (see also 
\cite{Pan}, \S3.3), for the definition of hyperlogarithm functions in genus zero, which relies on a decomposition result for the shuffle algebra of 
a vector space (see Lem.~\ref{lem:panzer:0805}). 

We show that both the ``tangential base point'' and the ``shuffle''  
regularisation procedures amount to the construction of algebra morphisms $k_{\overrightarrow 0},k^\sha_0 : 
\mathrm{Sh}(V)\to\mathcal O_{hol}(\tilde{\mathcal E}_S)$ 
(\S\ref{subsect:4:1:0805}, Lem. \ref{lem410:0805} and \S\ref{subsect:4:2:0805}, Lem.~\ref{lem:approche:panzer}) where~$V$ is a suitable 
vector space, which we show to coincide (\S\ref{subsect:4:2:0805}, Lem.~\ref{lem:identif:approches}).
In~\S\ref{subsect:tildeGamma:0805} we then apply the construction of these two equal morphisms to the definition of the functions 
$\GLarg{n_1 &n_2 &\ldots &n_r}{a_1 &a_2 &\ldots &a_r}{-}$ and to the proof of their 
properties. 

\subsection{Tangential base point regularisation}\label{subsect:4:1:0805}

For each $z_0\in\tilde{\mathcal E}_S$, the linear map $I_{z_0}$ defined in \eqref{def:I:x0} extends to a unique linear map  
$$
I_{z_0} : \mathrm{Sh}(\Omega_{hol}(\tilde{\mathcal E}_S))\to\mathcal O_{hol}(\tilde{\mathcal E}_S)
$$
satisfying the identities $I_{z_0}(a)(z_0)=\varepsilon(a)$ and $d(I_{z_0}([\omega_1|\cdots|\omega_n])=
I_{z_0}([\omega_1|\cdots|\omega_{n-1}])\cdot \omega_n$ for any $n\geq0$ and $\omega_1,\ldots,\omega_n\in
\Omega_{hol}(\tilde{\mathcal E}_S)$; 
then $I_{z_0}$ is an algebra morphism, given by iterated integration based at $z_0$. 

  
\begin{defn}
    (a) $V$ is the complex vector space freely generated by the symbols $\left(\begin{smallmatrix}n\\a\end{smallmatrix}\right)$, 
    where $(n,a)\in\mathbb Z_{\geq 0}\times \tilde S$, where $\tilde S$ is as in Def. \ref{def:tildeS:3006}(b).

    (b) $V_+\subset V$ is the vector subspace generated by $\left(\begin{smallmatrix}n\\a\end{smallmatrix}\right)$, where $(n,a)\neq(1,0)$. 

    (c) $\mathrm{Sh}^*(V)$ is the subspace $\mathbb C1\oplus[V_+|\mathrm{Sh}(V)]$ of $\mathrm{Sh}(V)$. 
\end{defn}

One checks that $\mathrm{Sh}^*(V)$ is a subalgebra of $\mathrm{Sh}(V)$. 

\begin{defn}
(a) $\kappa : V\to\Omega_{hol}(\tilde{\mathcal E}_S)$ is the linear map induced by $\left(\begin{smallmatrix}n\\a\end{smallmatrix}\right)
\mapsto\omega_{n,a}$ for any $(n,a)\in\mathbb Z_{\geq 0}\times \tilde S$.

(b) We denote by $\kappa_* : \mathrm{Sh}(V)\to\mathrm{Sh}(\Omega_{hol}(\tilde{\mathcal E}_S))$ the induced Hopf algebra morphism. 

(c) For $z_0\in\tilde{\mathcal E}_S$, $k_{z_0} : \mathrm{Sh}(V)\to\mathcal O_{hol}(\tilde{\mathcal E}_S)$ 
is the composed algebra morphism $I_{z_0}\circ\kappa_*$.  
\end{defn}

We fix $\delta>0$ such that $]0,\delta[\subset \mathbb{C}\smallsetminus pr^{-1}(S)$; we fix 
a continuous map $\mathrm{sect} : \,]0,\delta[\to \tilde{\mathcal E}_S$ such that $\varpi\circ\mathrm{sect}$ is the canonical inclusion.  

\begin{defn}
$\mathcal F\subset \mathbb C^{]0,\delta[}$ is the set of functions $\phi : \,]0,\delta[\to\mathbb C$, such that for some $\alpha>0$ one has
$\phi(t)=O(t^{1-\alpha})$ when $t\to0$. 
\end{defn}

\begin{lem}\label{lem:alg:0805}
(a) The map $\mathbb C[X]\otimes\mathcal F\to\mathcal F$ given by $P\otimes \phi\mapsto P\bullet \phi:=(t\mapsto P(\log(t))\phi(t))$
is well-defined and makes $\mathcal F$ into a $\mathbb C[X]$-module.    

(b) The direct sum $\mathbb C[X]\oplus\mathcal F$, equipped with the product given by $(P,\phi)\cdot(Q,\psi)=(PQ,P\bullet \psi+Q\bullet \phi+\phi\psi)$, 
is a commutative and associative algebra. 

(c) The algebra from (b) is equipped with a pair of algebra morphisms
$$
\mathbb C\stackrel{\mathrm{ev}}{\leftarrow}\mathbb C[X]\oplus\mathcal F\stackrel{\mathrm{can}}{\hookrightarrow}\mathbb C^{]0,\delta[}
$$
where $\mathrm{ev}(P,\phi):=\phi(0)$ and $\mathrm{can}(P,\phi):=(t\mapsto P(\log(t))+\phi(t))$, 
and the algebra morphism $\mathrm{can}$ is injective. 
\end{lem}

\begin{proof}
(a) follows from the fact that for any $r\geq 0$, one has $(\log(t))^r=O(t^{-\alpha})$  at the neighborhood of $0$
for any $\alpha>0$. (b) is an immediate direct check. Let us prove (c). The fact that $\mathrm{ev}$ and 
$\mathrm{can}$ are algebra morphisms can be verified directly. Let us prove the injectivity of $\mathrm{can}$ (see also 
[BGF], Lem. 3.350). 
Assume that $(P,\phi)\in \mathbb C[X]\oplus\mathcal F$ and $\mathrm{can}(P,\phi)=0$. If $P$ is nonzero, let $d\geq 0$ be its degree and 
$a_dX^d$ be its dominant term. Then $\mathrm{can}(P,\phi)(t)\sim a_d(\log(t))^d$ at the neighborhood of $0$, which contradicts  
 $\mathrm{can}(P,\phi)=0$, therefore $P=0$. Then $\mathrm{can}(0,\phi)=\phi$, which implies $\phi=0$.  
\end{proof}

In what follows, we will view $\mathbb C[X]\oplus\mathcal F$ as a subalgebra of $\mathbb C^{]0,\delta[}$, 
making use of the injectivity of the morphism $\mathrm{can}$.  

\begin{lem}\label{lem:4:5:0805}
For $a\in\mathrm{Sh}(V)$ and $z\in\tilde{\mathcal E}_S$, denote by $k(a,z)$ the assignment 
$]0,\delta[\,\ni t\mapsto k_{\mathrm{sect}(t)}(a)(z)\in\mathbb C$; then $k(a,z)\in\mathbb C^{]0,\delta[}$.    
For any $z\in\tilde{\mathcal E}_S$, the map $a\mapsto k(a,z)$ is an algebra morphism $\mathrm{Sh}(V)\to \mathbb C^{]0,\delta[}$. 
\end{lem}

\begin{proof}
    This follows from the fact that for any $z\in\tilde{\mathcal E}_S$ and $t\in]0,\delta[$, the map 
    $a\mapsto  k_{\mathrm{sect}(t)}(a)(z)$ is an algebra morphism $\mathrm{Sh}(V)\to \mathbb C$ (see Def. 4.2(c)).  
\end{proof}

\begin{lem}\label{lem:4:7:0505}
For each $z\in\tilde{\mathcal E}_S$, there exists a complex number $G(z)$ such that the function  
$]0,\delta[\,\ni t\mapsto \int_{\mathrm{sect}(t)}^z\omega_{1,0}+\mathrm{log}(t)$
belongs to $G(z)+\mathcal F$. Then $G\in\mathcal O_{hol}(\tilde{\mathcal E}_S)$ and 
$(t\mapsto G(\mathrm{sect}(t)))\in(t\mapsto\mathrm{log}(t))
+\mathcal F$.  
\end{lem}

\begin{proof}
Since $\omega_{1,0}$ has a pole of order $1$ at $0$, for any $z\in\tilde{\mathcal E}_S$, the map 
$t\mapsto\int_{\mathrm{sect}(t)}^z \omega_{1,0}+\mathrm{log}(t)$ is an analytic function; denoting by $G(z)$ 
its value at 0, it therefore belongs to $G(z)+\mathcal F$. One then has $G(z)
=G(\mathrm{sect}(\delta/2))+\int_{\mathrm{sect}(\delta/2)}^z \omega_{1,0}$ 
for any $z\in\tilde{\mathcal E}_S$, which implies the holomorphicity statement.   
\end{proof} 

Notice that one has
\begin{equation}\label{decomp;shuffle}
  \mathrm{Sh}(V)=\oplus_{r\geq 0}[\underbrace{\left(\begin{smallmatrix}1\\0\end{smallmatrix}\right)|\ldots|
 \left(\begin{smallmatrix}1\\0\end{smallmatrix}\right)}_{r\text{ factors}}|\mathrm{Sh}^*(V)].   
\end{equation}

\begin{defn}
(a) The degree of an element $a\in \mathrm{Sh}(V)$ is $-\infty$ if $a=0$, and is the maximum of the
set of integers $r\geq 0$ such that the projection of $a$ in the $r$-th component of the decomposition \eqref{decomp;shuffle}
is nonzero if $a\neq 0$.    

(b) The degree of an element of $\mathbb C[X]\oplus\mathcal F$  is the polynomial degree of the first component. 
\end{defn}

\begin{lem}\label{lem:toto:0805}
(a) For $a\in\mathrm{Sh}(V)$ and $z\in\tilde{\mathcal E}_S$, one has $k(a,z)\in \mathbb C[X]\oplus\mathcal F$, and $\mathrm{deg}(k(a,z))\leq\mathrm{deg}(a)$. 

(b) For any $z\in\tilde{\mathcal E}_S$, the map $\mathrm{Sh(V)}\ni a\mapsto k(a,z)\in \mathbb C[X]\oplus\mathcal F$ is an algebra morphism. 
\end{lem}

\begin{proof}
(a) By \eqref{decomp;shuffle}, it suffices to prove that for any $r\geq 0$ 
 and $a\in [\underbrace{\left(\begin{smallmatrix}1\\0\end{smallmatrix}\right)|\ldots|
 \left(\begin{smallmatrix}1\\0\end{smallmatrix}\right)}_{r\text{ factors}}|a^*]$, with $a^*\in\mathrm{Sh}^*(V)$, one
 has $k(a,z)\in\mathbb C[X]_{\leq r}\oplus\mathcal F$. The element $k(a,z)\in\mathbb C^{]0,\delta[}$ is the map taking 
 $t\in\,]0,\delta[$ to the first term of the sequence of equalities
\begin{align*}
&I_{\mathrm{sect}(t)}([\underbrace{\omega_{1,0}|\ldots|\omega_{1,0}}_{r\text{ factors}}|\kappa(a^*)])(z)
=I_{\mathrm{sect}(t)}([\underbrace{dG|\ldots|dG}_{r\text{ factors}}|\kappa_*(a^*)])(z)
\\
&=\frac{1}{r!}\,I_{\mathrm{sect}(t)}\big([(G-G(\mathrm{sect}(t)))^r \odot \kappa_*(a^*)]\big)(z)
\\ & =\sum_{l=0}^r \frac{1}{l!(r-l)!}\,(-G(\mathrm{sect}(t))^l I_{\mathrm{sect}(t)}\big([G^{r-l}\odot \kappa_*(a^*)]\big)(z)     
\end{align*}
where $\odot$ is the linear map $\mathcal O_{hol}(\tilde{\mathcal E}_S)\otimes\mathrm{Sh}(\mathcal O_{hol}(\tilde{\mathcal E}_S))
\to\mathrm{Sh}(\mathcal O_{hol}(\tilde{\mathcal E}_S))$ given by $f\odot(a_1\otimes\cdots\otimes a_n):=(fa_1)\otimes\cdots\otimes a_n$. 
One shows that  for any $l\in[\![0,r]\!]$, the map $t\mapsto I_{\mathrm{sect}(t)}[G^{r-l}\odot \kappa_*(a^*)]$ belongs to 
$\mathbb C+\mathcal F$. 
It follows from Lem. \ref{lem:4:7:0505} that for any $l\in[\![0,r]\!]$, the function
$t\mapsto(-G(\mathrm{sect}(t))^l$ belongs to $(t\mapsto(\mathrm{log}(t))^l)+\mathcal F$. 
Since $((t\mapsto(\mathrm{log}(t))^l)+\mathcal F)\cdot (\mathbb C+\mathcal F)\subset 
\mathbb C(t\mapsto(\mathrm{log}(t))^l)+\mathcal F$, $k(a,z)$ belongs to 
$\sum_{l=0}^r \mathbb C(t\mapsto(\mathrm{log}(t))^l)+\mathcal F$, which is $\mathbb C[X]_{\leq r}+\mathcal F$. 

Statement (b) follows from (a), Lem. \ref{lem:4:5:0805} and the injectivity of $\mathrm{can}$ (see Lem. \ref{lem:alg:0805}(c)).  
\end{proof}

\begin{defn}
    For $a\in\mathrm{Sh}(V)$, $k_{\overrightarrow 0}(a)\in\mathbb C^{\tilde{\mathcal E}_S}$ is the function $\tilde{\mathcal E}_S\to \mathbb C$ given by 
    $z\mapsto k_{\overrightarrow 0}(a)(z):=\mathrm{ev}(k(a,z))$. 
\end{defn}

\begin{lem}\label{lem410:0805}
(a) The map $\mathrm{Sh}(V)\to \mathbb C^{\tilde{\mathcal E}_S}$ given by $a\mapsto k_{\overrightarrow 0}(a)$ is an algebra morphism. 

(b) For $z\in\tilde{\mathcal E}_S$, $a\in\mathrm{Sh}(V)$, 
 one has 
 $$
k(a,z)=k(a^{(1)},\mathrm{sect}(\delta/2))
 k_{\mathrm{sect}(\delta/2)}(a^{(2)})(z),
 $$
where we denote $\Delta(a)$ as $a^{(1)}\otimes a^{(2)}$ (equality in $\mathbb C[X]\oplus \mathcal F$).  

(c) For any $a\in\mathrm{Sh}(V)$, $k_{\overrightarrow 0}(a)$ belongs to the subalgebra $\mathcal O_{hol}(\tilde{\mathcal E}_S)\subset \mathbb C^{\tilde{\mathcal E}_S}$.  

(d) The map $k_{\overrightarrow 0} : \mathrm{Sh}(V)\to\mathcal O_{hol}(\tilde{\mathcal E}_S)$ is an algebra morphism. 
\end{lem}

\begin{proof}
(a) For each $z\in\tilde{\mathcal E}_S$, the composition $\mathrm{Sh}(V)\stackrel{k(-,z)}{\to} \mathbb C[X]\oplus\mathcal F\stackrel{\mathrm{ev}}{\to}\mathbb C$
is an algebra morphism as it is a composition of such morphisms (see Lem. \ref{lem:toto:0805}(b) and Lem. \ref{lem:alg:0805}(c)). The map 
$k_{\overrightarrow 0} : \mathrm{Sh}(V)\to \mathbb C^{\tilde{\mathcal E}_S}$
is the composition $\mathrm{Sh}(V)\to \mathrm{Sh}(V)^{\tilde{\mathcal E}_S}\to \mathbb C^{\tilde{\mathcal E}_S}$, where the first map is the diagonal 
morphism and the second map is the product $\prod_{z\in\tilde{\mathcal E}_S}(\mathrm{ev}\circ k(-,z))$. Since both the diagonal morphism and 
$\mathrm{ev}\circ k(-,z)$ are algebra morphisms for each $z\in\tilde{\mathcal E}_S$, $k_{\overrightarrow 0}$ is an algebra morphism, which proves (a). 

(b) Both sides belong to $\mathbb C[X]\oplus \mathcal F$, which is contained in $\mathbb C^{]0,\delta[}$, so it suffices the prove that they coincide as functions on $]0,\delta[$. For $t\in\,]0,\delta[$, one has 
\begin{align*}
&k(a,z)(t)=k_{\mathrm{sect}(t)}(a)(z)=k_{\mathrm{sect}(t)}(a^{(1)})(\mathrm{sect}(\delta/2))
k_{\mathrm{sect}(\delta/2)}(a^{(2)})(z)
\\ & =k(a^{(1)},\mathrm{sect}(\delta/2))(t)\cdot k_{\mathrm{sect}(\delta/2)}(a^{(2)})(z)
\end{align*}
where the first and last equalities follow from definitions, and the middle equality follows from~\cite{EZ2}, Lem.~2.5.  

(c) Applying $\mathrm{ev}$ to the identity of (b), for any $a\in\mathrm{Sh}(V)$ and any $z\in \tilde{\mathcal E}_S$ one finds
$$
k_{\overrightarrow 0}(a)(z)=k_{\overrightarrow 0}(a^{(1)})(\mathrm{sect}(\delta/2))
 k_{\mathrm{sect}(\delta/2)}(a^{(2)})(z)
$$
(equality in $\mathbb C$). Since, for any $b\in\mathrm{Sh}(V)$, the map $z\mapsto k_{\mathrm{sect}(\delta/2)}(a^{(2)})(z)$
belongs to $\mathcal O_{hol}(\tilde{\mathcal E}_S)\subset\mathbb C^{\tilde{\mathcal E}_S}$, it follows that 
the map $z\mapsto k_{\overrightarrow 0}(a)(z)$ also belongs to $\mathcal O_{hol}(\tilde{\mathcal E}_S)$. 

(d) This statement follows from (a) and (c).
\end{proof}

\subsection{Shuffle regularisation}\label{subsect:4:2:0805}

We now present the construction of the algebra morphism $k_0^\sha : \mathrm{Sh}(V) \to\mathcal O_{hol} (\tilde{\mathcal E}_S)$. 
In this approach, Lem. \ref{lem:alg:0805} is replaced by: 

\begin{lem}\label{lem:diag:3006}
Let us equip $\mathbb C\oplus\mathcal F$ with the algebra structure such that $1\in\mathbb C$ is the unity and~$\mathcal F$ is a nonunital subalgebra. 
There is a diagram of algebra morphisms
    $$
        \mathbb C\stackrel{\mathrm{ev}_0}{\leftarrow}\mathbb C\oplus\mathcal F\stackrel{\mathrm{can}_0}{\hookrightarrow}\mathbb C^{]0,\delta[}
    $$
where $\mathrm{ev}_0$ is the projection $\mathbb C\oplus\mathcal F\to\mathbb C$ and $\mathrm{can}_0(\lambda,\phi)=\lambda+\phi$.   
\end{lem}

\begin{proof}
Obvious.
\end{proof}

\begin{rem}
The diagram from Lem. \ref{lem:diag:3006} is the pull-back of the diagram $\mathbb C\leftarrow C([0,\delta[)
\hookrightarrow\mathbb C^{]0,\delta[}$ by the corestriction 
$\mathbb C\oplus\mathcal F\to C([0,\delta[)$ of $\mathrm{can}_0$, 
where $C([0,\delta[)$ is the space of continuous complex functions on  
$[0,\delta[$, where $C([0,\delta[)\to\mathbb C$ is the evaluation at $0$ and where $C([0,\delta[)\hookrightarrow\mathbb C^{]0,\delta[}$
is the natural inclusion. 
\end{rem}

One also uses the following: 
\begin{lem}[see \cite{Pan}, Lem. 3.2.4] \label{lem:see:pa}\label{lem:panzer:0805}
There exists a unique algebra isomorphism $\mathrm{Sh}^*(V)[X]\to\mathrm{Sh}(V)$ given by $a\mapsto a$ for $a\in\mathrm{Sh}^*(V)$ 
and $X\mapsto \left(\begin{smallmatrix}1\\0\end{smallmatrix}\right)$. 
\end{lem}




\begin{lemdef}\label{lem:approche:panzer}
(a) For $a\in \mathrm{Sh}^*(V)$ and $z\in\tilde{\mathcal E}_S$, the function $]0,\delta[\ni t\mapsto k_{\mathrm{sect}(t)}(a)(z)$
belongs to $\mathbb C\oplus\mathcal F$; denote by $k_0^\sha(a)(z)\in\mathbb C$ its image by $\mathrm{ev}_0$. Then, for any 
$z\in\tilde{\mathcal E}_S$,
$a\mapsto (t\mapsto k_{\mathrm{sect}(t)}(a)(z))$ is an algebra morphism $\mathrm{Sh}^*(V)\to \mathbb C\oplus\mathcal F$, and 
$a\mapsto k^\sha_0(a)(z)$ is an algebra morphism $\mathrm{Sh}^*(V)\to \mathbb C$. 

(b) For $a\in \mathrm{Sh}^*(V)$, the function $z\mapsto k^\sha_0(a)(z)$ belongs to $\mathcal O_{hol}(\tilde{\mathcal E}_S)$; 
we denote by $k^\sha_0 : \mathrm{Sh}^*(V)\to \mathcal O_{hol}(\tilde{\mathcal E}_S)$ the map taking $a$ to this function. Then $k^\sha_0$
is an algebra morphism.  

(c) There is a unique algebra morphism $k^\sha_0 : \mathrm{Sh}(V)\to\mathcal O_{hol}(\tilde{\mathcal E}_S)$ that extends 
$k^\sha_0 : \mathrm{Sh}^*(V)\to\mathcal O_{hol}(\tilde{\mathcal E}_S)$ and such that 
$\left(\begin{smallmatrix}1\\0\end{smallmatrix}\right)\mapsto G$.
\end{lemdef}

\begin{proof} (a) The first statement follows from the specialisation of Lem. \ref{lem:toto:0805}(a) for $\mathrm{deg}(a)\leq 0$.
The second statement follows from the fact that for any pair $(z,t)$, the map $a\mapsto k_{\mathrm{sect}(t)}(a)(z)$
is an algebra morphism $\mathrm{Sh}^*(V)\to\mathbb C$, and from the first statement; the last statement follows from the second statement
by composing with $\mathrm{ev}_0$. 

(b) Let $a\in \mathrm{Sh}^*(V)$. Since $\Delta(\mathrm{Sh}^*(V))\subset \mathrm{Sh}^*(V)\otimes \mathrm{Sh}(V)$, 
the equality from Lem. \ref{lem410:0805}(b) takes place in $\mathbb C\oplus\mathcal F$. Applying $\mathrm{ev}_0$ to it, one obtains that for every $z\in\tilde{\mathcal E}_S$
$$
k_0^\sha(a)(z)=k_0^\sha(a^{(1)})(\mathrm{sect}(\delta/2))
 k_{\mathrm{sect}(\delta/2)}(a^{(2)})(z)
$$
(equality in $\mathbb C$). The first part of (b) then follows from the fact that 
for any $a\in\mathrm{Sh}(V)$, the map $z\mapsto k_{\mathrm{sect}(\delta/2)}(a^{(2)})(z)$
belongs to $\mathcal O_{hol}(\tilde{\mathcal E}_S)\subset\mathbb C^{\tilde{\mathcal E}_S}$. The second part of (b) follows from its first part, 
combined with statement that the maps $a\mapsto (z\mapsto k_0^\sha(a))$ is an algebra morphism 
$\mathrm{Sh}^*(V)\to \mathbb C^{\tilde{\mathcal E}_S}$, which follows from (a). 

(c) This follows from (b) and from Lem. \ref{lem:see:pa}. 
\end{proof}

\begin{prop}\label{prop:2:1}\label{lem:identif:approches}
One has $k_{\overrightarrow 0}=k_0^\sha$. 
\end{prop}

\begin{proof} One checks that the restrictions of $k_{\overrightarrow 0}$ and $k_0^\sha$ to $\mathrm{Sh}^*(V)$ coincide. 
Let us compute $k_{\overrightarrow 0}(\left(\begin{smallmatrix}1\\0\end{smallmatrix}\right))$. 
For  $z\in\tilde{\mathcal E}_S$, $k(\left(\begin{smallmatrix}1\\0\end{smallmatrix}\right),z)$ 
is the function $]0,\delta[\, \ni t\mapsto k_{\mathrm{sect}(t)}(\left(\begin{smallmatrix}1\\0\end{smallmatrix}\right))
(z)=\int_{\mathrm{sect}(t)}^z \omega_{1,0}$. By Lem.~\ref{lem:4:7:0505}, this function belongs to 
$(t\mapsto-\mathrm{log}(t)+G(z))+\mathcal F=P(\mathrm{log}(t))+\mathcal F$, where 
$P(X)=G(z)-X\in\mathbb C[X]$. It follows that $k_{\overrightarrow 0}(\left(\begin{smallmatrix}1\\0\end{smallmatrix}\right))(z)=G(z)=
k_0^\sha(\left(\begin{smallmatrix}1\\0\end{smallmatrix}\right))(z)$, hence that 
$k_{\overrightarrow 0}(\left(\begin{smallmatrix}1\\0\end{smallmatrix}\right))=
k_0^\sha(\left(\begin{smallmatrix}1\\0\end{smallmatrix}\right))$. By 
 Lem.~\ref{lem:see:pa}, this implies the statement. 
\end{proof}  

\begin{lem}\label{NEWLEMMA:2704}
(a) For any $z_0\in\tilde{\mathcal E}_S$ and any $a\in\mathrm{Sh}(V)$ one has the equality (in $\mathcal O_{hol}(\tilde{\mathcal E}_S)$)
\begin{equation}\label{rel:0905}
k_{\overrightarrow 0}(a)=k_{\overrightarrow 0}(a^{(1)})(z_0)\cdot k_{z_0}(a^{(2)}).
\end{equation}

(b) For any $z_0\in\tilde{\mathcal E}_S$, the images of the morphisms $k_{z_0}$ and $k_{\overrightarrow 0} : \mathrm{Sh}(V)\to\mathcal O_{hol}(\tilde{\mathcal E}_S)$ coincide. 
\end{lem}

\begin{proof} (a) Fix $a\in\mathrm{Sh}(V)$ and $z\in \tilde{\mathcal E}_S$. By \cite{EZ2}, Lem. 2.5, one has the identity 
$k_{\mathrm{sect}(t)}(a)(z)=k_{\mathrm{sect}(t)}(a^{(1)})(z_0)\cdot k_{z_0}(a^{(2)})(z)$ for any $t\in\,]0,\delta[$. 
One derives from this the equality (in $\mathbb C^{]0,\delta[}$) $(t\mapsto k_{\mathrm{sect}(t)}(a)(z))=(t\mapsto 
k_{\mathrm{sect}(t)}(a^{(1)})(z_0))
\cdot k_{z_0}(a^{(2)})(z)$. Since both sides actually belong to the subalgebra $\mathbb C[X]\oplus\mathcal F\subset \mathbb C^{]0,\delta[}$, then this may be viewed as an equality in $\mathbb C[X]\oplus\mathcal F$. Applying to both sides the morphism 
$\mathrm{ev} : \mathbb C[X]\oplus\mathcal F\to\mathbb C$, one derives the equality (in $\mathbb C$)
$k_{\overrightarrow 0}(a)(z)=k_{\overrightarrow 0}(a^{(1)})(z_0)\cdot k_{z_0}(a^{(2)})(z)$. 

It follows that, for any $a\in\mathrm{Sh}(V)$, one has $(z\mapsto k_{\overrightarrow 0}(a)(z))=(z\mapsto k_{\overrightarrow 0}(a^{(1)})(z_0))
\cdot k_{z_0}(a^{(2)})(z)$ (equality in $\mathbb C^{\tilde{\mathcal E}_S}$). Both sides actually belong to $\mathcal O_{hol}(\tilde{\mathcal E}_S)$, and 
are respectively equal to the two sides of \eqref{rel:0905}, which proves (a). 

(b) Eq. \eqref{rel:0905} implies the inclusion $k_{\overrightarrow 0}(\mathrm{Sh}(V)) \subset k_{z_0}(\mathrm{Sh}(V))$, and  
also implies the relation $k_{z_0}(a)=k_{\overrightarrow 0}(S(a^{(1)}))(z_0) \cdot k_{\overrightarrow 0}(a^{(2)})$ for any $a \in \mathrm{Sh}(V)$, 
where $S$ is the antipode of $\mathrm{Sh}(V)$, which implies the opposite inclusion $k_{z_0}(\mathrm{Sh}(V)) \subset 
k_{\overrightarrow 0}(\mathrm{Sh}(V))$.  \end{proof}

\subsection{The functions $\tilde\Gamma$}\label{subsect:tildeGamma:0805}

\begin{defn}[see \cite{BDDT}, eq. (4.13)] \label{def:3:2:1004} 
For any $r\geq 1$, any $n_1,\ldots ,n_r\geq 0$ and any $a_1,\ldots ,a_r\in \tilde S$, one defines the function 
$\GLarg{n_1 &n_2 &\ldots &n_r}{a_1 &a_2 &\ldots &a_r}{-}\in \mathcal{O}_{hol}(\tilde{\mathcal E}_S)$ by
\begin{equation*}
  \GLarg{n_1 &n_2 &\ldots &n_r}{a_1 &a_2 &\ldots &a_r}{-}\,:=\,
  k_{\overrightarrow 0}([\left(\begin{smallmatrix}n_1\\a_1\end{smallmatrix}\right)|\ldots|\left(\begin{smallmatrix}n_r\\a_r\end{smallmatrix}\right)]).
\label{defgam}
\end{equation*}
One also sets $\GLarg{}{}{-}:=1$ for $r=0$. 
\end{defn}

It follows from Prop. \ref{prop:2:1} that $\GLarg{1}{0}{-}=G$. 

\begin{lem}[see also \cite{Pan}, Lem. 3.3.4]
For any $t\in\mathrm{Sh}(V$) and $(n,a) \in \mathbb Z_{\geq 0}\times \tilde S$, one has  
$$
d(k_{\overrightarrow 0}([\,t\,| \left(\begin{smallmatrix}n\\a\end{smallmatrix}\right)])=k_{\overrightarrow 0}(t)\cdot \omega_{n,a}
$$
(equality in $\Omega_{hol}(\tilde{\mathcal E}_S)$).  
Equivalent formulation: for any $n_1,\ldots,n_r\geq 0$ and $a_1,\ldots,a_r\in \tilde S$, one has
\begin{equation}\label{DE:1104}
d(\GLarg{n_1 &n_2 &\ldots &n_{r}}{a_1 &a_2 &
\ldots &a_{r}}{-})=T_{a_r}(g_{n_r})\cdot \GLarg{n_1 &n_2 &\ldots &n_{r-1}}{a_1 &a_2 &
\ldots &a_{r-1}}{-}\cdot dz
\end{equation}

\end{lem}

\begin{proof} The identity 
\begin{equation}\label{toto:2702}
d(k_{\mathrm{sect}(\delta/2)}([\,t\,| \left(\begin{smallmatrix}n\\a\end{smallmatrix}\right)])=k_{\mathrm{sect}(\delta/2)}(t)\cdot \omega_{n,a}
\end{equation} 
for any $t$ and $(n,a)$ as above follows from \cite{EZ2}, eq. (2.1.1).
Let us define the map $\partial^V : \mathrm{Sh}(V)\to\mathrm{Sh}(V)\otimes V$ by setting
$\partial^V(1)=0$ and $\partial^V([t|v]):=t\,\otimes\,v$ for any $t\in\mathrm{Sh}(V)$ and $v\in V$. 
For $t\in\mathrm{Sh}(V)$, denote $\partial^V(t)$ by $t^{[1]} \otimes t^{[2]}$. Then \eqref{toto:2702} implies the identity 
\begin{equation}\label{titi:2702}
d(k_{\mathrm{sect}(\delta/2)}(t))=k_{\mathrm{sect}(\delta/2)}(t^{[1]})\cdot \omega(t^{[2]}),
\end{equation}
where $\omega : V\to \Omega_{hol}(\tilde{\mathcal E}_S)$ 
is the map given by $\left(\begin{smallmatrix}n\\a\end{smallmatrix}\right)\mapsto \omega_{n,a}$. 

Applying Lem. \ref{NEWLEMMA:2704} with $z_0:=\mathrm{sect}(\delta/2)$, one finds that,
for any $t\in\mathrm{Sh}(V)$,
\begin{equation}\label{TOTO:2702}
k_{\overrightarrow 0}(t)=k_{\overrightarrow 0}(t^{(1)})(\mathrm{sect}(\delta/2))\cdot k_{\mathrm{sect}(\delta/2)}(t^{(2)})
\end{equation} 
(equality in $\mathcal O_{hol}(\tilde{\mathcal E}_S)$), 
which gives rise to the first equality in 
\begin{align*}
& d(k_{\overrightarrow 0}(t))=k_{\overrightarrow 0}(t^{(1)})(\mathrm{sect}(\delta/2))\cdot d(k_{\mathrm{sect}(\delta/2)}(t^{(2)}))
\\
&=k_{\overrightarrow 0}(t^{(1)})(\mathrm{sect}(\delta/2))\cdot k_{\mathrm{sect}(\delta/2)}((t^{(2)})^{[1]}) \cdot\omega((t^{(2)})^{[2]})
\\ & =k_{\overrightarrow 0}((t^{[1]})^{(1)})(\mathrm{sect}(\delta/2))\cdot k_{\mathrm{sect}(\delta/2)}((t^{[1]})^{(2)}) \cdot\omega(t^{[2]})
=k_{\overrightarrow 0}(t^{[1]}) \cdot \omega(t^{[2]}), 
\end{align*}
where the second equality follows from \eqref{titi:2702}, the third equality follows from 
$(id\otimes\partial^V)\circ\Delta=(\Delta\otimes id) \circ \partial^V$, and the fourth equality 
follows from \eqref{TOTO:2702}. One derives $d(k_{\overrightarrow 0}(t))=k_{\overrightarrow 0}(t^{[1]}) \cdot \omega(t^{[2]})$, 
from which the claimed identity follows. 
\end{proof}

\section{Expression of a basis of the algebra $A_{\mathcal E_S}$ in terms of the functions $\tilde\Gamma$}
\label{sect:5:2606}

Recall the notation $O_S:=\mathcal{O}(\mathcal{E}_S)$ from \S\ref{NEWSECTION}. The purpose of this section is to show that the collection of functions $\tilde\Gamma$ 
gives rise to a basis of $A_{\mathcal E_S}$ over $O_S$ (Thm. \ref{THM:MAIN}), alternative to the family of 
elliptic HLs (see §\ref{NEWSECTION}). For this, we introduce an algebra $\mathcal G$ 
generated by the functions~$\tilde\Gamma$ and show its inclusion in $A_{\mathcal E_S}$ (§\ref{sect:5:1:2606}). In 
§§\ref{sect:diff:alg:1:2606} and \ref{sect:diff:alg:2:2606}, we carry out a precise analysis of differential algebras 
related to $\mathcal E_S$, which enables us in §\ref{sect:4:6:2704} to prove the equality $\mathcal G=A_{\mathcal E_S}$ 
(see Thm. \ref{THMEQ}). 
In §\ref{sect:4:5:1903}, we prove a linear independence result for the functions $\tilde\Gamma$ (see Prop. \ref{thm:main:bis}), 
based on the criterion of \cite{DDMS}, thereby proving Thm. \ref{THM:MAIN}. In §\ref{TODO}, 
we give examples of the relations implied by Thm. \ref{THM:MAIN} between the functions $\tilde\Gamma$ and the elliptic HLs from \S\ref{NEWSECTION}.


\subsection{The algebra $\mathcal G$ and its inclusion in $A_{\mathcal E_S}$}\label{sect:5:1:2606}

\begin{defn}\label{defn:mathcalG:2704}
$\mathcal{G}$ is the subalgebra of $\mathcal O_{hol}(\tilde{\mathcal E}_S)$ generated by $O_S[g_1]$ and by the functions 
$\GLarg{n_1 &n_2 &\ldots &n_r}{a_1 &a_2 &\ldots &a_r}{-}$, where $r\geq 0$, $n_1,\ldots ,n_r\geq 0$ and 
$a_1,\ldots ,a_r\in \tilde S$
(see Def. \ref{def:3:2:1004}), namely 
$$
\mathcal{G}:=O_S[g_1]\Big[\GLarg{n_1 &n_2 &\ldots &n_r}{a_1 &a_2 &\ldots &a_r}{-}\,|\,r\geq 0,\, n_1,\ldots ,n_r\geq 0,\, a_1,\ldots ,a_r\in \tilde S\Big]\,.
$$ 
\end{defn}
In this definition, the subalgebra $O_S[g_1]$ of $\mathcal O_{hol}(\mathbb C\smallsetminus pr^{-1}(S))$ 
is viewed as a subalgebra of $\mathcal O_{hol}(\tilde{\mathcal E}_S)$
via the morphism $\varpi^* : \mathcal O_{hol}(\mathbb C\smallsetminus pr^{-1}(S))\to\mathcal O_{hol}(\tilde{\mathcal E}_S)$. 

Recall now the filtration $F_\bullet^\delta\mathcal{O}_{hol}(\tilde{\mathcal E}_S)$ of the algebra 
$\mathcal{O}_{hol}(\tilde{\mathcal E}_S)$ defined in \S \ref{sect:2:2:2606}. 

\begin{lem}\label{lem:8:1704}
For any $n\geq 0$ and $\tilde s\in \tilde S$, one has $T_{\tilde s}g_n\in F_n^\delta\mathcal{O}_{hol}(\tilde{\mathcal E}_S)$.  
\end{lem}

\begin{proof}
Let $\tilde s\in\tilde S$. It follows from \eqref{diffpropgn} that the collection of functions $(T_{\tilde s}g_n)_{n\geq 0}$ satisfies 
the relation $d(T_{\tilde s}g_n)=\sum_{k=0}^{n-1}(-1)^{n-k}\,T_{\tilde s}g_k\cdot T_{\tilde s}(E_{n-k+1})\cdot dz$ for any $n\geq 1$. 
Let us prove the statement by induction on $n\geq 0$. For $n=0$, one has $T_{\tilde s}g_0=1\in 
F_0^\delta\mathcal{O}_{hol}(\tilde{\mathcal E}_S)$. Assume that $T_{\tilde s}g_k\in F_k^\delta\mathcal{O}_{hol}(\tilde{\mathcal E}_S)$ 
for $k<n$. Since $T_{\tilde s}E_{l}\in\mathcal O(\mathcal E_S)$ for $l\geq 2$, one has $T_{\tilde s}(E_l)\cdot 
dz \in \Omega(\mathcal E_S)$ for these values of $l$, hence $d(T_{\tilde s}g_n)=\sum_{k=0}^{n-1}(-1)^{n-k}T_{\tilde s}g_k
\cdot T_{\tilde s}(E_{n-k+1})\cdot dz\in F_{n-1}^\delta\mathcal{O}_{hol}(\tilde{\mathcal E}_S)\cdot
p^*\Omega(\mathcal E_S)$, and therefore $T_{\tilde s}g_n\in F_n^\delta\mathcal{O}_{hol}(\tilde{\mathcal E}_S)$. 
\end{proof}

\begin{prop}\label{Prop230125}
For any $r\geq 0$, any $n_1,\ldots ,n_r\geq 0$ and any $a_1,\ldots ,a_r\in \tilde S$, the function $\GLarg{n_1 &n_2 &\ldots &n_r}{a_1 &a_2 &\ldots &a_r}{-}$ belongs to $F^\delta_{r+n_1+\cdots +n_r}\mathcal{O}_{hol}(\tilde{\mathcal E}_S)$.
\end{prop}
\begin{proof}
Let us prove the statement by induction on $r$. If $r=0$ the statement is clear. If $r>0$, assume the statement to be true until $r-1$, 
so that $\GLarg{n_1 &n_2 &\ldots &n_{r-1}}{a_1 &a_2 &\ldots &a_{r-1}}{-}$ belongs to 
$F^\delta_{r-1+n_1+\cdots +n_{r-1}}\mathcal{O}_{hol}(\tilde{\mathcal E}_S)$.
By Lem. \ref{lem:8:1704}(b), one has $T_{a_r}g_{n_r}\in F_{n_r}^\delta\mathcal{O}_{hol}(\tilde{\mathcal E}_S)$. The facts that 
$F_\bullet^\delta\mathcal{O}_{hol}(\tilde{\mathcal E}_S)$ is an algebra filtration and that $dz\in p^*\Omega(\mathcal E_S)$ then 
imply that the right-hand side of eq. \eqref{DE:1104} belongs to 
$F^\delta_{r-1+n_1+\cdots +n_r}\mathcal{O}_{hol}(\tilde{\mathcal E}_S)\cdot p^*\Omega(\mathcal E_S)$, which implies that 
$\GLarg{n_1 &n_2 &\ldots &n_r}{a_1 &a_2 &\ldots &a_r}{-}$ belongs to $F^\delta_{r+n_1+\cdots +n_r}\mathcal{O}_{hol}(\tilde{\mathcal E}_S)$. 
This concludes the proof by induction.
\end{proof}

\begin{cor}\label{cor:1704}
    One has $\mathcal G\subset A_{\mathcal E_S}$.  
\end{cor}

\begin{proof}
One has $O_S=\mathcal O(\mathcal E_S)\subset F^\delta_1\mathcal O(\tilde{\mathcal E}_S)\subset F^\delta_\infty
\mathcal O(\tilde{\mathcal E}_S)$. Moreover, Lem. \ref{lem:8:1704} with $\tilde s=0$ and $n=1$ implies 
$g_1\in F^\delta_1\mathcal O(\tilde{\mathcal E}_S) \subset F^\delta_\infty\mathcal O(\mathcal E_S)$. Together with 
Prop. \ref{Prop230125} and the fact that $F^\delta_\infty\mathcal O(\mathcal E_S)$ is an algebra, this implies that 
$\mathcal G\subset F^\delta_\infty\mathcal O(\mathcal E_S)$. 
By \cite{EZ2}, Thm.~C, one has $F^\delta_\infty\mathcal O(\mathcal E_S)=A_{\mathcal E_S}$, which implies the result. 
\end{proof}

Prop. \ref{Prop230125}, together with the functional characterisation of the algebra 
$F_\infty^\delta \mathcal{O}_{hol}(\tilde{\mathcal E}_S)$ 
in Thm.~C from~\cite{EZ2}, implies that the functions $\GLarg{n_1 &n_2 &\ldots &n_r}{a_1 &a_2 &\ldots &a_r}{-}$ have moderate growth 
at the cusps in the sense explained in \cite{EZ2}, \S3.1, 
and unipotent monodromy in the sense of the introduction. 
 

\subsection{The differential algebra $\mathcal O(\mathcal E_0)[g_1]$}\label{sect:diff:alg:1:2606}

Let us set $\mathcal E_0:=\mathcal E_{\{pr(0)\}}$ $(=\mathcal E\smallsetminus\{pr(0)\})$, 
and $O:=O_{\{pr(0)\}}$ $(=\mathcal O(\mathcal E_0))$. Then $O$ is  
the algebra of regular functions on $\mathcal E_0$, which coincides with the subalgebra of $\Lambda$-invariant 
functions in $\mathcal O_{mer}(\mathbb C,\Lambda)$. The derivation 
$\partial=\partial/\partial z$ of $\mathcal O_{mer}(\mathbb C,\Lambda)$ restricts to a derivation of $O$.  

It is well-known that a linear $\mathbb C$-basis of $O$ is $\{1,\wp^{(k)}\,|\,k\geq0\}$. It follows from 
\eqref{relations:wp:Er} that an alternative linear $\mathbb C$-basis of $O$ is 
$\{E_n\,|\,n\in\mathbb Z_{\geq0}\smallsetminus\{1\}\}$, where we set $E_0:=1$. 
Let us define an algebra filtration on $O$ by the order of poles at $pr(0)$. One then
has $F_nO:=\sum_{k\in[\![0,n]\!]\smallsetminus\{1\}}\mathbb CE_k$. The associated graded algebra $\mathrm{gr}(O)$ is isomorphic to the 
subalgebra $\mathbb C\oplus Y^2\mathbb C[Y]$ of the polynomial algebra $\mathbb C[Y]$, where $Y$ has degree~1; 
more precisely, one has $\mathrm{gr}_n(O)\simeq \mathbb C\cdot Y^n$ for any $n\in\mathbb Z_{\geq0}\smallsetminus\{1\}$, 
with $\overline E_n\simeq Y^n$. 

Let us equip $\mathbb C[X]$ with the degree filtration, given by  $F_n\mathbb C[X]=\oplus_{k\in[\![0,n]\!]}\mathbb CX^k$ for any $n\geq0$; then
$\mathrm{gr}(\mathbb C[X])\simeq\mathbb C[X]$, where $X$ has degree 1. 

Moreover, let us equip the tensor product algebra $O\otimes\mathbb C[X]$ with the tensor product filtration, so that
$F_n(O\otimes\mathbb C[X])=\sum_{k=0}^n F_kO\otimes F_{n-k}\mathbb C[X]$. 
The associated graded algebra $\mathrm{gr}(O\otimes\mathbb C[X])$ is then isomorphic to $\mathrm{gr}(O)\otimes\mathrm{gr}(\mathbb C[X])$, 
which is isomorphic to the graded subalgebra $\mathbb C[X]\oplus Y^2\mathbb C[X,Y]$ of $\mathbb C[X,Y]$, where $X,Y$ have degree 1. 

Recall that $g_1=E_1\in \mathcal O_{mer}(\mathbb C,\Lambda)$, and $g_1'=-E_2$. Let $O[g_1]$
be the subalgebra of $\mathcal O_{mer}(\mathbb C,\Lambda)$ generated by $O$ and $g_1$. 
It follows from $g_1'=-E_2$ that the subalgebra  $O[g_1]$ of $\mathcal O_{mer}(\mathbb C,\Lambda)$
is stable under $\partial$. 

\begin{lem}\label{NEWLEM}
    One has $O[g_1] \cap \mathcal O_{hol}(\mathbb C)=\mathbb C$. 
\end{lem}  

\begin{proof} Let us prove inductively on $d\geq 0$ that if $a_0,\ldots,a_d\in O$ and $\sum_i a_ig_1^i\in \mathcal O_{hol}(\mathbb C)$, then 
$a_0 \in \mathbb C$ and $a_1=\ldots=a_d=0$. If $d=0$, then $a_0 \in O \cap\mathcal O_{hol}(\mathbb C)=\mathbb C$, which proves the statement for $d=0$. 
Assume the statement at step $d-1$ and let us prove it at step $d$. Let $a_0,\ldots,a_d \in O$ be such that $f:=\sum_i a_ig_1^i \in \mathcal O_{hol}(\mathbb C)$. 
Then 
\begin{align*}
\sum_{j=0}^{d-1}g_1^j\bigg(\sum_{i>j}a_i\binom{i}{j}(2\pi\mathrm i)^{i-j}\bigg)&=\sum_{(i,j)|j<i}a_i{i \choose j}g_1^j(2\pi\mathrm i)^{i-j}=\sum_i a_i\big((g_1+2\pi\mathrm i)^i-g_1^i\big)\\
&=(T_\tau-id)(f)\in \mathcal O_{hol}(\mathbb C),
\end{align*}
which by the induction assumption implies that $\sum_{i>0}a_i(2\pi\mathrm i)^{i} \in \mathbb C$ and for any $j=1,\ldots,d-1$,  
$\sum_{i>j}a_i {i \choose j}(2\pi\mathrm i)^{i-j}=0$. This implies $a_2=\ldots=a_d=0$ and $a_1 \in \mathbb C$. Then 
$a_0+a_1g_1 \in \mathcal O_{hol}(\mathbb C)$. Assume $a_1 \neq 0$. Since $g_1$ has a simple pole at $0$ with residue $1$, and since $a_0+a_1g_1 \in \mathcal O_{hol}(\mathbb C)$, 
$a_0$ has a 
simple pole at $0$ with residue $-a_1$, which contradicts $a_0 \in O$. Therefore $a_1=0$. Then $a_0\in\mathcal O_{hol}(\mathbb C)\cap O=\mathbb C$, which
implies the statement at step $d$.
\end{proof} 

\begin{lem}\label{lem:iso}
    The algebra morphism $O_S\otimes\mathbb C[X]\to O_S[g_1]$ induced by $f\otimes1\mapsto f$ for $f\in O_S$ and 
    $1\otimes X\mapsto g_1$ is an isomorphism. 
\end{lem}

\begin{proof}
Let $P$ belong to the kernel of this morphism and let $d$ be its degree as a polynomial in~$X$. If $P\neq 0$, then there exist 
$f_0,\ldots,f_d\in O_S$ with $f_d\neq 0$, such that $P=\sum_{i=0}^d f_i \otimes X^i$. Then 
$\sum_{i=0}^d f_ig_1^i=0$. Applying $(T_\tau-id)^d$ to this equality, using the invariance of $f_i$ under $T_\tau$ and the relation 
$(T_\tau-id)^d(g_1^i)=\delta_{i,d}(2\pi\mathrm{i})^dd!$ for $i\in[\![0,d]\!]$, one obtains $f_d=0$, a contradiction.  
\end{proof}

Let us equip $O[g_1]$ with the image $F_\bullet(O[g_1])$ of the algebra filtration $F_\bullet(O\otimes\mathbb C[X])$ by the 
algebra isomorphism of Lem. \ref{lem:iso}, specialized to $S=\{0\}$. One then has 
\begin{equation}\label{titi:0204}
\forall n\geq 0,\quad F_n(O[g_1])=\sum_{k=0}^n (F_kO)\cdot g_1^{n-k}, \quad \mathrm{gr}_n(O[g_1])\simeq \mathbb C\cdot X^n\oplus Y^2\cdot \mathbb C[X,Y]_{n-2},  
\end{equation}
where the notation $\mathbb C[X,Y]_{n-2}$ indicates the subspace of $\mathbb C[X,Y]$ of polynomials which are homogeneous of degree $n-2$.

\begin{lem}\label{lem:1719}
  If $A$ is an algebra equipped with a filtration $F_\bullet A$ and if 
$(a_n)_{n\geq 1}$ is a sequence such that $a_n\in F_nA$ for $n\geq 1$, then: 

(a) the sequence $(b_n)_{n\geq 1}$
of elements of $A$ defined by $1+\sum_{n\geq 1}b_n\alpha^n=\mathrm{exp}(\sum_{n\geq 1}a_n\alpha^n)$ (equality in $A[[\alpha]])$) 
is such that $b_n\in F_nA$ for any $n\geq 1$;

(b) one has $1+\sum_{n\geq 1}\overline b_n=
\mathrm{exp}(\sum_{n\geq 1}\overline a_n)$ (equality in $\hat\oplus_{n\geq 0}\mathrm{gr}_n(A)$).   
\end{lem}

The combination of (b) with the specialisation of Lem. \ref{NEWLEM} to $S=\{0\}$ enables one to prove a 
strengthening of (b), namely the image of the family $(g_n)_{n\geq 0}$ is a $\mathbb C$-basis of the cokernel $O[g_1]/\partial(O[g_1])$. 

\begin{proof}
One has for any $n\geq 1$ the identity 
$$
b_n=\sum_{k\geq 1}\frac{1}{k!}\sum_{\substack{n_1+\cdots+n_k=n,\\(n_1,\ldots,n_k)\in\mathbb Z_{\geq1}^k}}
a_{n_1}\cdots a_{n_k},
$$
which together with the algebra filtration properties implies (a). This identity implies that, for any $n\geq 1$,
$$
\overline b_n=\sum_{k\geq 1}\frac{1}{k!}\sum_{\substack{n_1+\cdots+n_k=n,\\(n_1,\ldots,n_k)\in\mathbb Z_{\geq1}^k}}
\overline a_{n_1}\cdots \overline a_{n_k},
$$
where $\overline a_i,\overline b_i$ are the classes of $a_i,b_i$ in $\mathrm{gr}_i(A)$, 
from which one derives (b). 
\end{proof}

\begin{lem}\label{lem:5:8:0702}
(a) For every $n\geq 0$ one has $g_n\in F_n(O[g_1])$. Moreover, the degree~$n$ polynomial 
$(X-Y)^{n-1}(X+(n-1)Y)/n!$ (by convention equal to $1$ when $n=0$) 
belongs to $\mathbb C[X]\oplus Y^2\mathbb C[X,Y]$ and is equal to the image of $\overline g_n\in\mathrm{gr}_n(O[g_1])$ 
by the isomorphism in \eqref{titi:0204}. 

(b) The family $(g_n)_{n\geq0}$ of elements of $O[g_1]$ is $\mathbb C$-linearly independent. 

(c) 
There is a  direct sum decomposition
\begin{equation}\label{titi:0504}
 \mathrm{Span}_{\mathbb C}\{g_n\,|\,n\geq 0\}\oplus\partial(O[g_1])
=O[g_1];    
\end{equation}
in particular, $\{g_n\,|\,n\geq 0\}$ is a $\mathbb C$-basis of the cokernel of the endomorphism $\partial$ of $O[g_1]$. 
\end{lem}

\begin{proof}
Eq. \eqref{defgn} implies the identity
$$
1+\sum_{n\geq 1}g_n\,\alpha^{n}=\exp\Big((1\otimes X)\,\alpha-\sum_{r\geq2}\frac{(-1)^{r}}{r}\,((E_r-e_r)\otimes 1)\,\alpha^r\Big), 
$$
(equality in $(O\otimes \mathbb C[X])[[\alpha]]$) where for $n\geq1$, $g_n$ is identified with its image in $O\otimes \mathbb C[X]$. 
The first statement in (a) follows from this equality, combined with Lem. \ref{lem:1719}(a) and 
the relations $X\in F_1\mathbb C[X]$ and $E_r-e_r\in F_rO$ for any $r\geq2$, which imply
$1\otimes X\in F_1(O\otimes\mathbb C[X])$ and $(E_r-e_r)\otimes1\in F_r(O\otimes\mathbb C[X])$ for $r\geq2$. 
The second statement in (a) follows from Lem. \ref{lem:1719}(b), which implies that
$$
1+\sum_{n\geq 1}\overline g_n=\exp\Big(\overline{1\otimes X}-\sum_{r\geq2}\frac{(-1)^{r}}{r}\,\overline{(E_r-e_r)\otimes 1}\Big)
=\exp\Big(X-\sum_{r\geq2}\frac{(-1)^{r}}{r}Y^r\Big)=e^{X-Y}(1+Y),
$$
(equality in the subalgebra $\mathbb C[[X]]\oplus Y^2\mathbb C[[X,Y]]$ of $\mathbb C[[X,Y]]$), 
which in turn implies that $\overline g_n=(X-Y)^n/n!+Y(X-Y)^{n-1}/(n-1)!=(X-Y)^{n-1}(X+(n-1)Y)/n!$. 

(b) If $\sum_n\lambda_ng_n=0$ is a nontrivial linear dependence relation, let $i:=\mathrm{max}\{i|\lambda_i\neq0\}$. Then 
$0=\sum_n\lambda_ng_n\in F_i(O[g_1])$, and the image of this element in $\mathrm{gr}_i(O[g_1])$ is $\lambda_i\overline{g_i}$, 
which is nonzero, this yielding a contradiction.  

(c) Let $D$ be the endomorphism of $O\otimes\mathbb C[X]$ defined by $D(f\otimes P):=\partial f\otimes P-fE_2\otimes P'$. Then $D$
is a derivation of $O\otimes\mathbb C[X]$, which is interwined with $\partial$ under the specialization of the isomorphism 
from Lem. \ref{lem:iso} to $S=\{0\}$. 

The derivation $\partial$ of $O$ is of filtration degree 1, i.e. $\partial(F_kO)\subset F_{k+1}O$ for any $k\geq 0$. 
The associated graded endomorphism $\mathrm{gr}(\partial)$ is then a degree 1 derivation of $\mathbb C\oplus Y^2\mathbb C[Y]$. 
The image of the derivation $Y^2\partial/\partial Y$ of $\mathbb C[Y]$ is contained in $Y^2\mathbb C[Y]$, therefore $Y^2\partial/\partial Y$
induces a derivation of $\mathbb C\oplus Y^2\mathbb C[Y]$ of degree 1, which can be seen to coincide with $\mathrm{gr}(\partial)$.  
One deduces from this that the derivation $\partial \otimes \mathrm{id}$ of $O \otimes \mathbb C[X]$ is of filtration degree 1, 
and that its associated graded is the derivation $Y^2 \partial/\partial Y$ of 
$\mathbb C[X] \oplus Y^2 \mathbb C[X,Y]$. 

The endomorphism $f\otimes P\mapsto -E_2f\otimes P'$ is a derivation of $O\otimes\mathbb C[X]$, which is also of filtration degree 1, as 
$E_2$ has filtration degree 2 and $P\mapsto P'$ has filtration degree $-1$. The associated graded of the endomorphism  
$f\mapsto E_2f$ of $O$ (resp. $P\mapsto P'$ of $\mathbb C[X]$) is the endomorphism $f\mapsto Y^2f$ of 
$\mathbb C \oplus Y^2\mathbb C[Y]$ (resp. $\partial/\partial X$ of $\mathbb C[X]$), therefore
the associated graded of $f\otimes P\mapsto -E_2f\otimes P'$ is the degree 1 derivation of 
$\mathbb C[X]\oplus Y^2\mathbb C[X,Y]$ given by $-Y^2\partial/\partial X$. 

It follows that $D$ has filtration degree 1, and that its associated graded is the derivation of $\mathbb C[X] \oplus Y^2\mathbb C[X,Y]$
induced by the derivation $Y^2({\partial\over\partial Y}-{\partial\over\partial X})$ of $\mathbb C[X,Y]$. 

We now prove by induction on $k\geq 0$ the equality of subspaces of $O[g_1]$
\begin{equation}\label{toto:1035:0204}
F_{k+1}(O[g_1])=\partial(F_k(O[g_1]))+\mathrm{Span}_{\mathbb C}(1,g_1,\ldots,g_{k+1}).
\end{equation} 

One has $F_0(O[g_1])=\mathbb C$ and $F_1(O[g_1])=\mathrm{Span}_{\mathbb C}(1,g_1)$, which implies \eqref{toto:1035:0204} for $k=0$. 
Assume $k>0$ and $F_{k}(O[g_1])=\partial(F_{k-1}(O[g_1]))+\mathrm{Span}_{\mathbb C}(1,g_1,\ldots,g_{k})$, and let us show \eqref{toto:1035:0204}.
The inclusion of the right-hand side in the left-hand side is obvious, so one has to show that this inclusion is an equality. 

Recall that the space $\partial(F_k(O[g_1]))$ is a subspace of $F_{k+1}(O[g_1])$. It follows from the commutativity of 
$$
\xymatrix{
F_k(O[g_1])\ar^{\partial}[r]\ar[d]&\ar[d]F_{k+1}(O[g_1])\\\mathrm{gr}_k(O[g_1])\ar_{\mathrm{gr}_k(\partial)}[r]&\mathrm{gr}_{k+1}(O[g_1])
}
$$
and from the surjectivity of $F_k(O[g_1])\to\mathrm{gr}_k(O[g_1])$ that the image of this subspace in $\mathrm{gr}_{k+1}(O[g_1])$ is 
equal to the image of $\mathrm{gr}_k(\partial) : \mathrm{gr}_k(O[g_1])\to \mathrm{gr}_{k+1}(O[g_1])$, which is identified
with the map 
\begin{equation}\label{1119:0204}
Y^2({\partial\over\partial X}-{\partial\over\partial Y}) : \mathbb CX^k\oplus Y^2\cdot\mathbb C[X,Y]_{k-2}
\to\mathbb CX^{k+1}\oplus Y^2\cdot\mathbb C[X,Y]_{k-1}.     
\end{equation}
The image of \eqref{1119:0204} is obviously contained in $Y^2\cdot\mathbb C[X,Y]_{k-1}$, which induces a linear map  
\begin{equation}\label{1126:0204}
Y^2({\partial\over\partial X}-{\partial\over\partial Y}) : \mathbb CX^k\oplus Y^2\cdot\mathbb C[X,Y]_{k-2}
\to Y^2\cdot\mathbb C[X,Y]_{k-1}.     
\end{equation}
The kernel of the linear endomorphism $Y^2({\partial\over\partial X}-{\partial\over\partial Y})$ of $\mathbb C[X,Y]$ is $\mathbb C[X+Y]$, which 
implies that the kernel of the linear map $Y^2({\partial\over\partial X}-{\partial\over\partial Y}) : \mathbb C[X,Y]_k\to\mathbb C[X,Y]_{k+1}$ is the 
one-dimensional vector space $\mathbb C\cdot(X+Y)^k$. The kernel of \eqref{1126:0204} is the intersection of this vector space 
with the source of \eqref{1126:0204}, which is zero as the coefficient of $YX^{k-1}$ in $(X+Y)^k$ is nonzero; this implies that the map 
\eqref{1126:0204} is injective.  The source and target of \eqref{1126:0204} both have dimension $k$; the equality of these dimensions, together with 
the injectivity of \eqref{1126:0204}, implies that \eqref{1126:0204} is a linear isomorphism. All this implies that the image of \eqref{1119:0204} is 
equal to $Y^2\cdot\mathbb C[X,Y]_{k-1}$, therefore that the image of $\partial(F_k(O[g_1]))$ by $F_{k+1}(O[g_1])\to\mathrm{gr}_{k+1}(O[g_1])\simeq 
\mathbb CX^{k+1}\oplus Y^2\cdot\mathbb C[X,Y]_{k-1}$ is $Y^2\cdot\mathbb C[X,Y]_{k-1}$. 

On the other hand, $\mathrm{Span}_{\mathbb C}(1,g_1,\ldots,g_k)\subset F_k(O[g_1])$, which implies that the image of 
 $\mathrm{Span}_{\mathbb C}(1,g_1,\ldots,g_{k+1})$ under $F_{k+1}(O[g_1])\to\mathrm{gr}_{k+1}(O[g_1])$ is $\mathbb C\,\overline g_{k+1}$, 
 whose image in $\mathbb CX^{k+1}\oplus Y^2\cdot\mathbb C[X,Y]_{k-1}$ is given by (a). 

Then the image of $\partial(F_k(O[g_1]))+\mathrm{Span}_{\mathbb C}(1,g_1,\ldots,g_{k+1})$ under 
$F_{k+1}(O[g_1])\to\mathrm{gr}_{k+1}(O[g_1])$ is the sum of the images of 
 $\partial(F_k(O[g_1]))$ and $\mathrm{Span}_{\mathbb C}(1,g_1,\ldots,g_{k+1})$ under this map, which is equal to the 
 sum $Y^2\cdot\mathbb C[X,Y]_{k-1}+\mathbb C(X-Y)^{k}(X+kY)/(k+1)!$, that can be simplified into 
$\mathbb C\cdot X^k\oplus Y^2\cdot\mathbb C[X,Y]_{k-1}\simeq\mathrm{gr}_{k+1}(O[g_1])$. 

It follows that the image of $\partial(F_k(O[g_1]))+\mathrm{Span}_{\mathbb C}(1,g_1,\ldots,g_{k+1})$ under 
the projection $F_{k+1}(O[g_1])\to\mathrm{gr}_{k+1}(O[g_1])$ is equal to its target. The space 
$\partial(F_k(O[g_1]))+\mathrm{Span}_{\mathbb C}(1,g_1,\ldots,g_{k+1})$ also contains 
$\partial(F_{k-1}(O[g_1]))+\mathrm{Span}_{\mathbb C}(1,g_1,\ldots,g_{k})$, which by the induction 
hypothesis is equal to $F_k(O[g_1])$, and therefore also to the kernel of the projection  $F_{k+1}(O[g_1])\to\mathrm{gr}_{k+1}(O[g_1])$. 
All this implies the equality  $\partial(F_k(O[g_1]))+\mathrm{Span}_{\mathbb C}(1,g_1,\ldots,g_{k+1})=F_{k+1}(O[g_1])$, which 
proves the induction step. This proves \eqref{toto:1035:0204} for any $k\geq0$. 

Let us now prove by induction on $k\geq 0$ the equality of subspaces of $F_{k+1}(O[g_1])$
\begin{equation}\label{intersection:0204}
\partial(F_k(O[g_1]))\cap \mathrm{Span}_{\mathbb C}(1,g_1,\ldots,g_{k+1})=\{0\}.
\end{equation} 
One has $F_0(O[g_1])=\mathbb C$ hence $\partial(F_0(O[g_1]))=0$, which implies \eqref{intersection:0204}
for $k=0$. Assume $k>0$ and $\partial(F_{k-1}(O[g_1]))\cap \mathrm{Span}_{\mathbb C}(1,g_1,\ldots,g_{k})=\{0\}$, and let us 
show \eqref{intersection:0204}. Let $P\in F_k(O[g_1])$, $(\lambda_i)_{i\in [0,k+1]}\in\mathbb C^{k+2}$ be such that 
$\partial(P)=\sum_i\lambda_i g_i$. The image of this equality under the projection $F_{k+1}(O[g_1])\to\mathrm{gr}_{k+1}(O[g_1])$
is $\lambda_{k+1}\overline{g_{k+1}}=\mathrm{gr}_k(\partial)(\overline P)$, where $\overline P$ is the image of $P$ in 
$\mathrm{gr}_k(O[g_1])$. The map $\mathrm{gr}_k(\partial) : \mathrm{gr}_k(O[g_1])\to\mathrm{gr}_{k+1}(O[g_1])$ is injective
(as this maps can be identified with \eqref{1126:0204} which has been proved to be injective), 
and its image does not contain $\overline{g_{k+1}}$ (as this image has been proved to be identified to 
$Y^2\mathbb C[X,Y]_{k-1}$ and in view of the identification of $\overline{g_{k+1}}$ in (a)). It follows that 
$\lambda_{k+1}=0$ and $\overline P=0$, therefore in the equality $\partial(P)=\sum_i\lambda_i g_i$ the left-hand side 
belongs to $\partial(F_{k-1}(O[g_1]))$ and the right-hand side to $\mathrm{Span}_{\mathbb C}(1,g_1,\ldots,g_{k})$, which 
by the induction assumption implies $\partial(P)=0$ and $\lambda_0=\ldots=\lambda_k=0$. This proves the induction step, therefore 
\eqref{intersection:0204} holds for any $k\geq 0$. 

Then 
\begin{align}\label{sochi:0702}
&\nonumber O[g_1]=\sum_{k \geq 0}F_{k+1}(O[g_1])
=\sum_{k \geq 0}\partial(F_k(O[g_1]))+\mathrm{Span}_{\mathbb C}(1,g_1,\ldots,g_{k+1})
\\ &\nonumber =\sum_{k \geq 0}\partial(F_k(O[g_1]))
+\sum_{k \geq 0}\mathrm{Span}_{\mathbb C}(1,g_1,\ldots,g_{k+1})
=\partial(\sum_{k \geq 0}F_k(O[g_1]))+
\mathrm{Span}_{\mathbb C}(g_i,i\geq 0)
\\ & =\partial(O[g_1])
+\mathrm{Span}_{\mathbb C}(g_i,i\geq0),
\end{align}
where the second equality follows from \eqref{toto:1035:0204}. 

For any $k,l \geq 0$, one has 
$$ 
\partial(F_kO[g_1]) \cap \mathrm{Span}_{\mathbb C}(1,g_1,\ldots,g_{l+1})
\subset 
\partial(F_{\mathrm{max}(k,l)}O[g_1]) \cap \mathrm{Span}_{\mathbb C}(1,g_1,\ldots,g_{\mathrm{max}(k,l)+1})=0,
$$ 
where the last equality follows from \eqref{intersection:0204}. This implies 
\begin{equation}\label{toto:0702}
\forall k,l\geq0,\quad \partial(F_kO[g_1]) \cap \mathrm{Span}_{\mathbb C}(1,g_1,\ldots,g_{l+1})=0.
\end{equation}
Then 
\begin{align}\label{sochibis:0702}
& \nonumber\partial(O[g_1]) 
\cap \mathrm{Span}_{\mathbb C}(g_i,i\geq 0)
=(\cup_{k \geq 0}\partial(F_kO[g_1]))
\cap (\cup_{l \geq 0}\mathrm{Span}_{\mathbb C}(1,g_1,\ldots,g_{l+1})
\\ & =\cup_{k,l \geq 0} \partial(F_kO[g_1]) \cap 
\mathrm{Span}_{\mathbb C}(1,g_1,\ldots,g_{l+1})=0,
\end{align}
where the last equality follows from \eqref{toto:0702}. 

The statement then follows from the combination of \eqref{sochi:0702} and
\eqref{sochibis:0702}. 
\end{proof}

\subsection{The differential algebra $\mathcal O(\mathcal E_S)[g_1]$}\label{sect:diff:alg:2:2606}

Let us denote by $\mathbb C(\mathcal E)$ the field of rational functions on~$\mathcal E$. As an algebra, it can be identified 
with the algebra of $\Lambda$-invariant meromorphic functions on~$\mathbb C$. Recall that, for $a\in\mathbb C$, the translation 
$T_a$ is an automorphism of the algebra of meromorphic functions on $\mathbb C$; it induces an automorphism of $\mathbb C(\mathcal E)$, 
which moreover depends only on the class of $a$ in $\mathcal E$. 

Recall the map $s\mapsto\tilde s$ from Def. \ref{def:tildeS:3006}(a). 

\begin{lem}\label{lem:toto:0304}
    (a) For any $s\in S$, one has $(T_{\tilde s}-id)(g_1)\in O_S$. 

    (b) $O_S=\sum_{s\in S}T_s(O)+\sum_{s\in S\smallsetminus\{pr(0)\}}\mathbb C\cdot (T_{\tilde s}-id)(g_1)$ (equality of subspaces of $\mathbb C(\mathcal E)$). 
\end{lem}

\begin{proof}
The relations $T_1g_1=g_1$ and $T_\tau g_1=g_1+2\pi\mathrm{i}$, together with the commutativity of the various $T_a$, $a\in\mathbb C$, implies 
that for any $s\in S$ the function $(T_{\tilde s}-1)(g_1)$ is $\Lambda$-invariant. It has simple poles at $0$ and $s$, which implies that 
it belongs to $O_S$, which proves (a). One clearly has $T_s(O)\subset O_S$ for any $s\in S$, so that (a) implies the inclusion 
$O_S\supset\sum_{s\in S}T_s(O)+\sum_{s\in S\smallsetminus\{pr(0)\}}\mathbb C\cdot (T_{\tilde s}-1)(g_1)$. If now $f\in O_S$, let $(f_s)_{s\in S}$
be the collection where $f_s\in\mathbb C((z-\tilde s))$ is the local expansion of $f$ at $\tilde s$. The map 
$O\to\mathbb C((z))\to\mathbb C((z))/z^{-1}\mathbb C[[z]]$ is surjective, so for any $s\in S$ one can find $o_s\in O$ whose image by this map coincides with the 
image of $f_s$ by the map $\mathbb C((z-\tilde s))\simeq\mathbb C((z))\to\mathbb C((z))/z^{-1}\mathbb C[[z]]$. Then $g:=f-\sum_{s\in S}T_s(o_s)$
belongs to $O_S$ and has at most simple poles at $S$.  Then $g-\sum_{s\in S\smallsetminus pr(0)}\mathrm{res}_s(g\cdot dz)\cdot(T_{\tilde s}-id)(g_1)$
belongs to $O_S$ and is regular on $S\smallsetminus pr(0)$; hence it belongs to $O$ and has at most a simple pole at $pr(0)$, which implies that it is constant. 
This proves the desired opposite inclusion. 
\end{proof}

Recall that $\mathcal O_{mer}(\mathbb C, pr^{-1}(S))$ is the algebra of meromorphic functions on $\mathbb C$ with set of poles contained in $pr^{-1}(S)$. 
Then $\Lambda$ acts by translation on $\mathcal O_{mer}(\mathbb C, pr^{-1}(S))$, and $O_S=\mathcal O_{mer}(\mathbb C, 
pr^{-1}(S))^\Lambda$. 
The operator $\partial=d/dz$ defines a derivation on $\mathcal O_{mer}(\mathbb C, pr^{-1}(S))$, which restricts to a derivation of $O_S$. 

Consider the subalgebra $O_S[g_1]$ of $\mathcal O_{mer}(\mathbb C, pr^{-1}(S))$. Then $\partial(g_1)=-E_2\in O_S$ implies that 
$O_S[g_1]$ is stable under $\partial$. 

\begin{lem}\label{lem:0904} Let $a\in\mathbb C\smallsetminus\Lambda$.  

(a) If $f\in\mathcal O_{mer}(\mathbb C, \Lambda \cup (a+\Lambda))$ is such that $T_1(f)=f$ and 
$(T_\tau-id)(f) \in O[g_1]+T_a(O[g_1])$, then $f \in O[g_1]+T_a(O[g_1])$. 

(b) If $f\in\mathcal O_{mer}(\mathbb C, \Lambda \cup (a+\Lambda))$ is such that $T_1(f)=f$ and that for some $n\geq 1$ one has 
$(T_\tau-id)^n(f) \in O[g_1]+T_a(O[g_1])$, then $f \in O[g_1]+T_a(O[g_1])$. 

(c) For any $n\geq 0$, one has 
\begin{equation}\label{theta:ids}
g_1^n\cdot T_a(g_1)\subset O[g_1]+T_a(O[g_1]).     
\end{equation}
\end{lem}

\begin{proof}
(a) 
The automorphism $T_\tau$ of 
$\mathcal O_{mer}(\mathbb C,\Lambda \cup (a+\Lambda))$ restricts to automorphisms of the subalgebras 
$O[g_1]$ and $T_a(O[g_1])$, therefore to an automorphism of the vector subspace $O[g_1]+T_a(O[g_1])$. 

The $O$-algebra morphism $O[X]\to O[g_1]$ induced by $X \mapsto g_1$ intertwines the linear endomorphism $T_\tau-id$ of $O[g_1]$ with the linear endomorphism 
$\mathrm{exp}(2\pi\mathrm{i}\partial_X)-1$ of $O[X]$, where $\partial$ is the derivation of $O[X]$ given by $O\mapsto 0$ and $X\mapsto 1$. On the other hand,
$\mathrm{exp}(2\pi\mathrm{i}\partial_X)-1={\mathrm{exp}(2\pi\mathrm{i}\partial_X)-1\over\partial_X}\circ\partial_X$, where 
${\mathrm{exp}(2\pi\mathrm{i}\partial_X)-1\over\partial_X}$ is a linear automorphism of $O[X]$ as $\partial_X$ is locally nilpotent, and 
$\partial_X$ is surjective. It follows that the linear endomorphism  $\mathrm{exp}(2\pi\mathrm{i}\partial_X)-1$ of $O[X]$ is surjective. 
It follows from this surjectivity, from the interwtining of this operator with $T_\tau-id$ and from the surjectivity of the map $O[X]\to O[g_1]$ 
that the linear endomorphism $T_\tau-id$ of $O[g_1]$ is surjective. This implies the surjectivity of the linear endomorphism $T_\tau-id$ of 
$T_a(O[g_1])$, and therefore that of the linear endomorphism $T_\tau-id$ of $O[g_1]+T_a(O[g_1])$.  

Let now $f$ be as in the (a). It follows from the surjectivity of the linear endomorphism $T_\tau-id$ of $O[g_1]+T_a(O[g_1])$ that there exists 
$\tilde f \in O[g_1]+T_a(O[g_1])$, such that $(T_\tau-id)(\tilde f)=(T_\tau-id)(f)$ (equality in $O[g_1]+T_a(O[g_1])$). Then $f-\tilde f$ belongs to 
$\mathcal O_{mer}(\mathbb C,\Lambda \cup (a+\Lambda))$ and is invariant both under $T_1$ and $T_\tau$, which implies that it belongs to 
$O_{\{pr(0),pr(a)\}}$, which by Lem. \ref{lem:toto:0304}(b) for $S=\{pr(0),pr(a)\}$ is equal to $O+T_a(O)+\mathbb C(T_a-id)(g_1)$, and is 
therefore contained in $O[g_1]+T_a(O[g_1])$. 
The statement then follows from $\tilde f\in O[g_1]+T_a(O[g_1])$ and $f-\tilde f\in O[g_1]+T_a(O[g_1])$.  

(b) The proof is by induction on $n\geq 1$. For $n=1$, the statement follows from (a). Assume the statement at step $n$ and let us prove it at step 
$n+1$. Let $g\in\mathcal O_{mer}(\mathbb C,\Lambda \cup (a+\Lambda))$ be such that $T_1(g)=g$ and  
$(T_\tau-id)^{n+1}(g) \in O[g_1]+T_a(O[g_1])$. Then $f:=(T_\tau-id)(g)$ is such that $T_1(f)=f$ as $T_1$ commutes
with $T_\tau-id$, and $(T_\tau-id)^n(f) \in O[g_1]+T_a(O[g_1])$. By the induction assumption, this implies 
$f\in O[g_1]+T_a(O[g_1])$. Then one has $T_(g)=g$ and $(T_\tau-id)(g)\in O[g_1]+T_a(O[g_1])$, which by (a) implies 
$g\in O[g_1]+T_a(O[g_1])$. This proves the statement at step $n+1$. 

(c) This follows from (b) and from the invariance of $g_1^n\cdot T_a(g_1)$ under $T_1$, which follows from that of $g_1$ and of $T_a(g_1)$, and from 
$(T_\tau-id)^{n+2}(g_1^n\cdot T_a(g_1))=0$. 
\end{proof}

\begin{lem}\label{lem:11:0504}
    One has $O_S[g_1]=\sum_{s\in S}T_{\tilde s}(O[g_1])$ (equality of subspaces of $\mathcal O_{mer}(\mathbb C, pr^{-1}(S))$). 
\end{lem}

\begin{proof}
For $s\in S$, one has $T_{\tilde s}(O)\subset O_S\subset O_S[g_1]$, where the first inclusion follows from 
Lem.~\ref{lem:toto:0304}(b). Moreover, $T_{\tilde s}(g_1)=g_1+(T_{\tilde s}-id)(g_1)\in O_S[g_1]$
by Lem. \ref{lem:toto:0304}(a). Then $T_{\tilde s}(O[g_1])$ is the subalgebra of $\mathcal O_{mer}(\mathbb C, pr^{-1}(S))$ 
generated by $T_{\tilde s}(O)$ and $T_{\tilde s}(g_1)$, which are both 
contained in $O_S[g_1]$, which implies the inclusion $T_{\tilde s}(O[g_1])\subset O_S[g_1]$ as $O_S[g_1]$ is an algebra. 
As $O_S[g_1]$ is stable under summation, this implies the inclusion $O_S[g_1]\supset\sum_{s\in S}T_{\tilde s}(O[g_1])$. 

In order to show the opposite inclusion, let us first show that $\sum_{s\in S}T_{\tilde s}(O[g_1])$ is a subalgebra of 
$\mathcal O_{mer}(\mathbb C, pr^{-1}(S))$. For each $s\in S$, the subspace $T_{\tilde s}(O[g_1])$ is a subalgebra of 
$\mathcal O_{mer}(\mathbb C,  pr^{-1}(S))$. So it is enough to prove that for any $s\neq t\in S$, one has 
$T_{\tilde s}(O[g_1])\cdot T_{\tilde t}(O[g_1])\subset T_{\tilde s}(O[g_1])+T_{\tilde t}(O[g_1])$. By translation invariance, 
it suffices to prove this for $t=0$, in which case $\tilde t=0$, i.e. to prove 
\begin{equation}\label{to:prove}
O[g_1]\cdot T_a(O[g_1])\subset O[g_1]+T_a(O[g_1])    
\end{equation}
for $a\in\mathbb C\smallsetminus\Lambda$. 

One has $O\cdot T_a(O)\subset O_{\{pr(0),pr(a)\}}=O+T_a(O)+\mathbb C\cdot(T_a-id)(g_1)$ where the first
inclusion follows from the fact that both $O$ and $T_a(O)$ are contained in $O_{\{pr(0),pr(a)\}}$ and that the latter set is an algebra
and the second inclusion follows from Lem. \ref{lem:toto:0304}(b) for $S=\{pr(0),pr(a)\}$. This implies $O\cdot T_a(O)\subset O[g_1]+T_a(O[g_1])$. 
Now $O[g_1]\cdot T_a(O[g_1])$ is the sub-$\mathbb C[g_1,T_a(g_1)]$-module of $\mathcal O_{mer}(\mathbb C, pr^{-1}(S))$
generated by $O\cdot T_a(O)$, therefore in order to prove \eqref{to:prove} it suffices to show that 
$O[g_1]+T_a(O[g_1])$ is a sub-$\mathbb C[g_1,T_a(g_1)]$-module of $\mathcal O_{mer}(\mathbb C,  pr^{-1}(S))$, 
i.e. is stable under multiplication by $g_1$ and $T_a(g_1)$. The inclusions $O[g_1]\cdot g_1\subset O[g_1]$ and 
$T_a(O[g_1])\cdot T_a(g_1)\subset T_a(O[g_1])$ are obvious, it therefore remains to prove  
\begin{equation}\label{to:prove:13200504}
 O[g_1]\cdot T_a(g_1)\subset O[g_1]+T_a(O[g_1]),\quad
 T_a(O[g_1])\cdot g_1\subset O[g_1]+T_a(O[g_1]). 
\end{equation}
 
For $f\in O$, one has $f\cdot T_a(g_1)=f(a)T_a(g_1)+(f-f(a))(T_a-id)(g_1)+(f-f(a))g_1$. 
Then $f(a)T_a(g_1)\in\mathbb CT_a(g_1)$, $(f-f(a))g_1\in O[g_1]$ and $(f-f(a))(T_a-id)(g_1)\in O$
as $f-f(a)$ belongs to $O$ and vanishes at $a$, while $(T_a-id)(g_1)$ belongs to $O_{\{0,a\}}$ and has only a simple pole at $a$. 
It follows that
\begin{equation}\label{toto:0504}
 O\cdot T_a(g_1)\subset O[g_1]+\mathbb CT_a(g_1).    
\end{equation}

Then one has $Og_1^n\cdot T_a(g_1)\subset O[g_1]g_1^n+\mathbb CT_a(g_1)g_1^n\subset O[g_1]+T_a(O[g_1])$
by virtue of \eqref{toto:0504} and \eqref{theta:ids}. This proves the first inclusion in \eqref{to:prove:13200504}, 
the second inclusion is a consequence of it (applying $T_{-a}$ and replacing $a$ by $-a$). 

This ends the proof of the fact that $\sum_{s\in S}T_{\tilde s}(O[g_1])$ is a subalgebra of 
$\mathcal O_{mer}(\mathbb C,pr^{-1}(S))$.

One has $O_S=\sum_{s\in S}T_s(O)+\sum_{s\in S\smallsetminus\{pr(0)\}}(T_{\tilde s}-id)(g_1)\subset \sum_{s\in S}T_{\tilde s}(O[g_1])$, 
where the equality follows from Lem. \ref{lem:toto:0304}(b) and the inclusion follows from $(T_{\tilde s}-id)(g_1)\in 
O[g_1]+T_{\tilde s}(O[g_1])$ for any $s\in S\smallsetminus\{pr(0)\}$. One also has 
$g_1\in T_0(O[g_1])\subset \sum_{s\in S}T_{\tilde s}(O[g_1])$. These two statements, together with 
the fact that $\sum_{s\in S}T_{\tilde s}(O[g_1])$ is an algebra, implies the inclusion 
$O_S[g_1]\subset\sum_{s\in S}T_{\tilde s}(O[g_1])$.
\end{proof}

\begin{lem}\label{lem:13:1004}
The vector subspace $\sum_{s\in S}T_{\tilde s}(\mathrm{Span}_{\mathbb C}\{g_n\,|\,n\geq 0\})\subset O_S[g_1]$ 
is such that 
$$
\sum_{s\in S}T_{\tilde s}(\mathrm{Span}_{\mathbb C}\{g_n\,|\,n\geq 0\})+\partial(O_S[g_1])
=O_S[g_1]; 
$$
in other words, the cokernel of the endomorphism $\partial$ of $O_S[g_1]$ is linearly spanned 
by the image of the family $\{T_{\tilde s}(g_n)\,|\,(s,n)\in S\times\mathbb Z_{\geq 0}\}$. 
\end{lem}

\begin{proof}
One has 
\begin{align*}
&O_S[g_1]=\sum_{s\in S}T_{\tilde s}(O[g_1])
=\sum_{s\in S}T_{\tilde s}(\partial(O[g_1])+\mathrm{Span}_{\mathbb C}\{g_n\,|\,n\geq 0\})
\\ & 
=\sum_{s\in S}T_{\tilde s}(\partial(O[g_1]))+\sum_{s\in S}T_{\tilde s}(\mathrm{Span}_{\mathbb C}\{g_n\,|\,n\geq 0\})
=\partial(\sum_{s\in S}T_{\tilde s}(O[g_1]))+\sum_{s\in S}T_{\tilde s}(\mathrm{Span}_{\mathbb C}\{g_n\,|\,n\geq 0\})
\\ & 
=\partial(O_S[g_1])+\sum_{s\in S}T_{\tilde s}(\mathrm{Span}_{\mathbb C}\{g_n\,|\,n\geq 0\}), 
\end{align*}
where the first and last equalities follow from Lem. \ref{lem:11:0504}, the second equality follows from eq.~\eqref{titi:0504}, the third equality 
follows from the linearity of $T_{\tilde s}$, and the fourth equality follows from the commutativity of $\partial$ and 
$T_{\tilde s}$.  
\end{proof}

\begin{lem}\label{NEWNEWLEM}
The sum map 
$$
\mathbb C1 \oplus \big(\bigoplus_{s \in S}T_{\tilde s}\mathrm{Span}\{g_n|n>0\}\big)  \oplus \partial(O_S[g_1])\to O_S[g_1] 
$$
is a vector space isomorphism. 
\end{lem}

\begin{proof}
One has 
\begin{align*}
&O_S[g_1]=\sum_{s \in S}T_{\tilde s}(O[g_1])
=\sum_{s\in S}T_{\tilde s}(\mathbb C1+\mathrm{Span}\{g_n|n>0\}+\partial(O[g_1])) 
\\ & =\mathbb C1+\sum_{s \in S}T_{\tilde s}(\mathrm{Span}\{g_n|n>0\})
+\partial(\sum_{s\in S} T_{\tilde s}(O[g_1]))   
=\mathbb C1+\sum_{s \in S}T_{\tilde s}(\mathrm{Span}\{g_n|n>0\})+\partial(O_S[g_1]) 
\end{align*}
where the second equality follows from Lem. \ref{lem:5:8:0702}(c), the third equality follows from the commutativity 
of $T_{\tilde s}$
with $\partial$, the fourth equality follows from Lem. \ref{lem:11:0504}; this implies that the said sum map is surjective. 

Let us prove its injectivity. Let $P \in O_S[g_1]$  
and $S \ni s \mapsto t_s \in \mathrm{Span}\{g_n|n>0\}$ and $\lambda \in \mathbb C$ be such that 
$$
\lambda+\partial(P)+\sum_{s \in S}T_{\tilde s}(t_s)=0. 
$$
By Lem. \ref{lem:11:0504}, there exists a map $S \ni s \mapsto P_s \in O[g_1]$ such that 
$P=\sum_{s \in S}T_{\tilde s}(P_s)$. Then 
$$
\lambda+\sum_{s \in S}T_{\tilde s}(\partial(P_s)+t_s)=0
$$
(equality in $O_S[g_1] \subset \mathcal O_{mer}(\mathbb C,p^{-1}(S))$). 

For each $s \in \tilde S$, the equality $T_{\tilde s}(\partial(P_s)+t_s)=-\lambda-\sum_{s'\in S\smallsetminus\{s\}}T_{\tilde s'}(\partial(P_s')+t_s')$ implies that the set of poles of $T_{\tilde s}(\partial(P_s)+t_s)$ 
is contained in $p^{-1}(s) \cap p^{-1}(S\smallsetminus\{s\})$, which is empty; so 
$T_{\tilde s}(\partial(P_s)+t_s)$ is an element of $O[g_1]$ which is holomorphic on $\mathbb C$, therefore is 
constant by Lem.~\ref{NEWLEM}. Therefore there exists a map $S \ni s \mapsto \lambda_s \in \mathbb C$ 
such that $T_{\tilde s}(\partial(P_s)+t_s)=\lambda_s$  for any $s \in S$, and therefore 
$$
\lambda+\sum_{s \in S}\lambda_s=0.   
$$
Then $\partial(P_s)=T_{\tilde s}^{-1}(\lambda_s)-t_s=\lambda_s-t_s$ for any $s \in S$. By 
Lem. \ref{lem:5:8:0702}(c), one has for any $s \in S$, 
$t_s-\lambda_s=0$ and $\partial(P_s)=0$. Since $t_s\in \mathrm{Span}\{g_n|n>0\}$ and by Lem. \ref{lem:5:8:0702}(b), 
one derives $t_s=\lambda_s=0$. Therefore $\partial(P)=\sum_s \partial(P_s)=0$, 
$\lambda=\sum_s \lambda_s=0$, which implies the injectivity.  
\end{proof}

\subsection{The equality $\mathcal G=A_{\mathcal E_S}$}\label{sect:4:6:2704}

\begin{lem}\label{lem:5:14:0603}
    $\mathcal G$ is equal to the $O_S[g_1]$-submodule of $\mathcal O_{hol}(\tilde{\mathcal E}_S)$ generated by the functions 
$\GLarg{n_1 &n_2 &\ldots &n_r}{a_1 &a_2 &\ldots &a_r}{-}$, where $r\geq 0$, $n_1,\ldots ,n_r\geq 0$ and 
$a_1,\ldots ,a_r\in\tilde S$, namely 
$$
\mathcal G=\sum_{\substack{r\geq 0,\, n_1,\ldots ,n_r\geq 0,\\ a_1,\ldots ,a_r\in \tilde S}}O_S[g_1]\cdot \GLarg{n_1 &n_2 &\ldots &n_r}{a_1 &a_2 &\ldots &a_r}{-}. 
$$
\end{lem}

\begin{proof}
This follows from the shuffle identity satisfied by the functions $\tilde\Gamma$, which follows from the algebra morphism status of $k_{\overrightarrow 0}$. 
\end{proof}

Set $\mathrm{int}:=\mathrm{int}_{\omega_0}$, where $\omega_0:=dz$. Then $\mathrm{int}$ is the endomorphism of 
$\mathcal O_{hol}(\tilde{\mathcal E}_S)$ given by $f\mapsto [z\mapsto \int_{z_0}^z f\omega_0]$, where $z_0$ is fixed in 
$\tilde{\mathcal E}_S$. 

\begin{defn}
For $r\geq0$, let us set 
$$
F_r(\mathcal G):=\sum_{r'\leq r}\sum_{((n_1,a_1),\ldots,(n_{r'},a_{r'}))\in(\mathbb Z_{\geq0}\times\tilde S)^{r'}}
O_S[g_1]\cdot \GLarg{n_1 &n_2 &\ldots &n_{r'}}{a_1 &a_2 &
\ldots &a_{r'}}{-}. 
$$
\end{defn}
One checks that $F_\bullet(\mathcal G)$ is an algebra filtration of $\mathcal G$.

\begin{lem}\label{lem:b:1104}
For any $r,n,n_1,\ldots,n_r\geq 0$, $a_1,\ldots,a_r\in \tilde S$ and $s\in S$, one has 
$$
\mathrm{int}(T_{\tilde s}(g_n)\cdot \GLarg{n_1 &n_2 &\ldots &n_{r}}{a_1 &a_2 &
\ldots &a_{r}}{-})\in F_{r+1}(\mathcal G). 
$$
\end{lem}

\begin{proof}
Both $\mathrm{int}(T_{\tilde s}(g_n)\cdot \GLarg{n_1 &n_2 &\ldots &n_{r}}{a_1 &a_2 &
\ldots &a_{r}}{-})$ and $\GLarg{n_1 &n_2 &\ldots &n_{r}&n}{a_1 &a_2 &
\ldots &a_{r}&\tilde s}{-}-\GLarg{n_1 &n_2 &\ldots &n_{r}&n}{a_1 &a_2 &
\ldots &a_{r}&\tilde s}{z_0}$ are elements of $\mathcal O_{hol}(\tilde{\mathcal E}_S)$.
By \eqref{DE:1104} their images by $d : \mathcal O_{hol}(\tilde{\mathcal E}_S)\to\Omega_{hol}(\tilde{\mathcal E}_S)$ 
coincide, and they both vanish at $z_0$. It follows that  
\begin{equation}\label{eq:int:T*Gamma:1104}
\mathrm{int}(T_{\tilde s}(g_n)\cdot \GLarg{n_1 &n_2 &\ldots &n_{r}}{a_1 &a_2 &
\ldots &a_{r}}{-})=\GLarg{n_1 &n_2 &\ldots &n_{r}&n}{a_1 &a_2 &
\ldots &a_{r}&\tilde s}{-}-\GLarg{n_1 &n_2 &\ldots &n_{r}&n}{a_1 &a_2 &
\ldots &a_{r}&\tilde s}{z_0}. 
\end{equation}
This identity implies the statement. 
\end{proof}

\begin{lem}\label{lem:a:1104}
For any $r,n_1,\ldots,n_r\geq 0$ and $a_1,\ldots,a_r\in \tilde S$, one has 
$$
\mathrm{int}(\partial(O_S[g_1])\cdot \GLarg{n_1 &n_2 &\ldots &n_{r}}{a_1 &a_2 &
\ldots &a_{r}}{-})\subset F_{r}(\mathcal G). 
$$
\end{lem}

\begin{proof}
By induction on $r$. For $r=0$, and $f\in O_S[g_1]$, one has $\mathrm{int}(\partial(f))=f-f(z_0)\in F_0(\mathcal G)$. 
Assume the statement at step $r-1$. 
For $f\in O_S[g_1]$, $n_1,\ldots,n_r\geq 0$ and $a_1,\ldots,a_r\in \tilde S$, 
integration by parts together with \eqref{DE:1104} yields  
\begin{align*}
&\mathrm{int}(\partial(f)\cdot \GLarg{n_1 &n_2 &\ldots &n_{r}}{a_1 &a_2 &
\ldots &a_{r}}{-})
\\ & =f\cdot \GLarg{n_1 &n_2 &\ldots &n_{r}}{a_1 &a_2 &
\ldots &a_{r}}{-}-f(z_0)\cdot \GLarg{n_1 &n_2 &\ldots &n_{r}}{a_1 &a_2 &
\ldots &a_{r}}{z_0}-\mathrm{int}(f\cdot T_{a_r}(g_{n_r})\cdot \GLarg{n_1 &n_2 &\ldots &n_{r-1}}{a_1 &a_2 &
\ldots &a_{r-1}}{-}).  
\end{align*}
Note that $f\cdot \GLarg{n_1 &n_2 &\ldots &n_{r}}{a_1 &a_2 &
\ldots &a_{r}}{-}\in F_r(\mathcal G)$ and $f(z_0)\cdot \GLarg{n_1 &n_2 &\ldots &n_{r}}{a_1 &a_2 &
\ldots &a_{r}}{z_0}\in\mathbb C\subset F_0(\mathcal G)$. 
Moreover, $f\in O_S[g_1]$ and $T_{a_r}(g_{n_r})\in O_S[g_1]$ by Lem. \ref{lem:13:1004}, which implies that 
$f\cdot T_{a_r}(g_{n_r})\in O_S[g_1]$.  By Lem.~\ref{lem:13:1004}, one may then decompose
$f\cdot T_{a_r}(g_{n_r})$ as $\partial(g)+\sum_{(n,a)\in\mathbb Z_{\geq0}\times\tilde S}
\lambda_{n,a}T_a(g_n)$, with $g\in O_S[g_1]$ and $\lambda_{n,a}\in\mathbb C$. 
Then 
\begin{align*}
&\mathrm{int}(f\cdot T_{a_r}(g_{n_r})\cdot \GLarg{n_1 &n_2 &\ldots &n_{r-1}}{a_1 &a_2 &
\ldots &a_{r-1}}{-})
\\ & =\mathrm{int}(\partial(g)\cdot \GLarg{n_1 &n_2 &\ldots &n_{r-1}}{a_1 &a_2 &
\ldots &a_{r-1}}{-})+\sum_{(n,a)\in\mathbb Z_{\geq0}\times\tilde S}
\lambda_{n,a} \mathrm{int}(T_a (g_n)\cdot \GLarg{n_1 &n_2 &\ldots &n_{r-1}}{a_1 &a_2 &
\ldots &a_{r-1}}{-}). 
\end{align*}
By the induction assumption, $\mathrm{int}(\partial(g)\cdot \GLarg{n_1 &n_2 &\ldots &n_{r-1}}{a_1 &a_2 &
\ldots &a_{r-1}}{-})$ belongs to $F_{r-1}(\mathcal G)$. Moreover, for each $(n,a)$, 
$ \mathrm{int}(T_a (g_n)\cdot \GLarg{n_1 &n_2 &\ldots &n_{r-1}}{a_1 &a_2 &
\ldots &a_{r-1}}{-})\in F_r(\mathcal G)$ by Lem. \ref{lem:b:1104}. 
All this implies $\mathrm{int}(\partial(f)\cdot \GLarg{n_1 &n_2 &\ldots &n_{r}}{a_1 &a_2 &
\ldots &a_{r}}{-})\in F_r(\mathcal G)$, which concludes the inductive argument. 
\end{proof}

\begin{prop}\label{prop:16:1104}
    For any $z_0\in\tilde{\mathcal E}_S$, the linear endomorphism $\mathrm{int}$ of $\mathcal O_{hol}(\tilde{\mathcal E}_S)$ 
    maps $\mathcal G$ to itself. 
\end{prop}

\begin{proof}
For $r\geq 0$, set 
$$
\mathrm{gr}_r(\mathcal G):=
\sum_{\substack{n_1,\ldots ,n_{r}\geq 0,\\ a_1,\ldots ,a_{r}\in \tilde S}}
O_S[g_1]\cdot \GLarg{n_1 &n_2 &\ldots &n_{r}}{a_1 &a_2 &
\ldots &a_{r}}{-}. 
$$
One then has $F_r(\mathcal G)=\sum_{r'\in[\![0,r]\!]}\mathrm{gr}_{r'}(\mathcal G)$. 
We will prove the inclusion $\mathrm{int}(F_r(\mathcal G))\subset F_{r+1}(\mathcal G)$ by induction on $r\geq0$. 
Let us first prove that $\mathrm{int}(F_0(\mathcal G))\subset F_1(\mathcal G)$. 
 
For $f\in O_S[g_1]$, one has $\mathrm{int}(\partial(f))=f-f(z_0)\in O_S[g_1]$, which 
implies that
\begin{equation}\label{1653:1004}
\mathrm{int}(\partial(O_S[g_1]))\subset O_S[g_1]\in F_0(\mathcal G).     
\end{equation}
It follows from \eqref{eq:int:T*Gamma:1104} with $r=0$ that, for $s\in S$ and $n\geq0$, one has 
$\mathrm{int}(T_{\tilde s}(g_n))=\GLarg{n}{\tilde s}{-}-\GLarg{n}{\tilde s}{z_0}$, 
which implies that
\begin{equation}\label{1656:1004}
    \mathrm{int}(T_{\tilde s}(g_n))\in F_1(\mathcal G).
\end{equation} 
Then 
\begin{align*}
& \mathrm{int}(F_0(\mathcal G))=\mathrm{int}(O_S[g_1])=\mathrm{int}(\partial(O_S[g_1])+\sum_{s\in S}T_{\tilde s}(\mathrm{Span}_{\mathbb C}\{g_n\,|\,n\geq0\}))
\\ & =\mathrm{int}(\partial(O_S[g_1]))+\sum_{s\in S,n\geq0}\mathbb C\cdot \mathrm{int}(T_{\tilde s}(g_n))
\subset F_0(\mathcal G)+\sum_{s\in S,n\geq0}F_1(\mathcal G)=F_1(\mathcal G),
\end{align*}
where the first equality follows from $F_0(\mathcal G)=O_S[g_1]$, 
the second equality follows from  Lem. \ref{lem:13:1004} and the inclusion follows from \eqref{1653:1004} and \eqref{1656:1004}. 
This proves the initial step of the induction. 

Assume that $r\geq 1$ and that $\mathrm{int}(F_{r-1}(\mathcal G))\subset F_r(\mathcal G)$. Let us prove that 
$\mathrm{int}(F_r(\mathcal G))\subset F_{r+1}(\mathcal G)$. One has $F_{r}(\mathcal G)=F_{r-1}(\mathcal G)
+\mathrm{gr}_r(\mathcal G)$, so by the induction assumption it suffices to prove that 
$\mathrm{int}(\mathrm{gr}_{r}(\mathcal G))\subset F_{r+1}(\mathcal G)$.

Combining Lem. \ref{lem:13:1004} with the definition of $\mathrm{gr}_{r}(\mathcal G)$, one finds  
$$
\mathrm{gr}_{r}(\mathcal G)=
\sum_{\substack{n_1,\ldots ,n_{r}\geq 0,\\ a_1,\ldots ,a_{r}\in \tilde S}}
\partial(O_S[g_1])\cdot \GLarg{n_1 &n_2 &\ldots &n_{r}}{a_1 &a_2 &
\ldots &a_{r}}{-}
+\sum_{n\geq 0,s\in S}\sum_{\substack{n_1,\ldots ,n_{r}\geq 0,\\ a_1,\ldots ,a_{r}\in \tilde S}}
\mathbb C\cdot T_{\tilde s}(g_n)\cdot \GLarg{n_1 &n_2 &\ldots &n_{r}}{a_1 &a_2 &
\ldots &a_{r}}{-},    
$$
which implies that
\begin{align*}
\mathrm{int}(\mathrm{gr}_{r}(\mathcal G))&=\sum_{\substack{n_1,\ldots ,n_{r}\geq 0,\\ a_1,\ldots ,a_{r}\in \tilde S}}
\mathrm{int}(\partial(O_S[g_1])\cdot \GLarg{n_1 &n_2 &\ldots &n_{r}}{a_1 &a_2 &
\ldots &a_{r}}{-})
\\ &+\sum_{n\geq 0,s\in S}\sum_{\substack{n_1,\ldots ,n_{r}\geq 0,\\ a_1,\ldots ,a_{r}\in \tilde S}}
\mathbb C\cdot \mathrm{int}(T_{\tilde s}(g_n)\cdot \GLarg{n_1 &n_2 &\ldots &n_{r}}{a_1 &a_2 &
\ldots &a_{r}}{-}),    
\end{align*}
It follows from Lem. \ref{lem:a:1104} (resp. Lem. \ref{lem:b:1104}) that the first (resp. second) summand of the 
right-hand side of this equality is contained in $F_r(\mathcal G)$ (resp. $F_{r+1}(\mathcal G)$), which implies that 
$\mathrm{int}(\mathrm{gr}_{r}(\mathcal G))\subset F_{r+1}(\mathcal G)$. 
\end{proof}

\begin{prop}\label{lem:17:1704}
The subspace $\mathcal G$ of $\mathcal O_{hol}(\tilde{\mathcal E}_S)$ is stable under the endomorphism $\mathrm{int}_\omega$ for any
$\omega \in \Omega(\mathcal E_S)$. 
\end{prop}

\begin{proof}
 Since $\omega_0=dz$ is a nowhere vanishing holomorphic differential, there exists 
 $f_\omega \in O_S$ such that $\omega=f_\omega\cdot \omega_0$. Then for $f \in \mathcal G$, one has 
 $\mathrm{int}_\omega(f)=\mathrm{int}(f_\omega\cdot f) \in \mathcal G$, where the equality follows from $\omega=f_\omega\cdot \omega_0$ and 
 the inclusion follows from $f_\omega\cdot f \in \mathcal G$, which stems from the fact that $\mathcal G$ is an $O_S$-module, and from 
 Prop. \ref{prop:16:1104}. 
\end{proof}

\begin{thm}\label{thm:main:2606}\label{THMEQ}
One has the equality  $\mathcal G=A_{\mathcal E_S}$. 
\end{thm}

\begin{proof}
Prop. \ref{lem:17:1704} implies that $\mathcal G$ is a subalgebra with unit of $\mathcal O_{hol}(\tilde{\mathcal E}_S)$
which is stable under the endomorphism $\mathrm{int}_\omega$ for any $\omega \in \Omega(\mathcal E_S)$. Together with 
Lem.-Def. \ref{lem:ez2:1704}, this implies that $\mathcal G \supset A_{\mathcal E_S}$. The statement follows from the combination of this inclusion with 
 Cor. \ref{cor:1704}. 
\end{proof}

\subsection{An $\mathcal O(\mathcal E_S)$-basis of $A_{\mathcal E_S}$ arising from the functions $\tilde\Gamma$}\label{sect:4:5:1903}

Since the ring $O_S$ 
is an integral domain, it injects into its fraction field, which is equal to the field $\mathbb C(\mathcal E)$ of
rational functions on $\mathcal E$ and is therefore independent of $S$, and will be denoted $K$. There is a unique extension of the derivation $\partial$ of $O_S$ to a derivation of~$K$, which will also be denoted $\partial$. 

\begin{lem}\label{LEMAB}
One has $O_S\cap \partial(K)=\partial(O_S)$ (equality of subspaces of $K$).
\end{lem}

\begin{proof}
For $f\in K\smallsetminus\{0\}$, let $\mathrm{Pole}(f) \subset \mathcal E$ be the set of its poles. 
If $f$ is nonconstant, then both~$f$ and $\partial(f)$ are in $K\smallsetminus\{0\}$, and
one has $\mathrm{Pole}(\partial f)=\mathrm{Pole}(f)$. On the other hand, 
$O_S=\{f \in K\smallsetminus\{0\}\,|\,\mathrm{Pole}(f) \subset S\} \cup \{0\}$. Let now $f\in K$ be such that 
$\partial(f) \in O_S$. If $\partial(f)=0$, then $f$ is constant and so $f \in O_S$.
If $\partial(f) \neq 0$, then $f$ is nonzero and $\mathrm{Pole}(f)=\mathrm{Pole}(\partial f) \subset S$, therefore 
$f \in O_S$. All this proves the implication $(f\in K\text{ and }\partial(f) \subset O_S)\implies(f\in O_S)$, which implies 
the inclusion $O_S \cap \partial(K)\subset \partial(O_S)$. The opposite inclusion is obvious.    
\end{proof}

Under the isomorphism $O_S[X]\simeq O_S[g_1]$, the derivation $\partial$ of $O_S[g_1]$ is intertwined with the 
derivation $\partial$ of $O_S[X]$, uniquely defined by the conditions that it extends the derivation~$\partial$ of~$O_S$ 
and that $\partial(X)=-E_2\subset O\subset O_S$. This derivation further extends to a derivation 
$\partial$ of $K[X]$, extending the derivation $\partial$ of $K$. 

\begin{lem}\label{MAIN POINT OF TODAY}
The natural map $O_S[X]/\partial(O_S[X])\to K[X]/\partial(K[X])$ is injective. 
\end{lem}

\begin{proof}
We first prove inductively on $n \geq 0$ the statement 
$$
(S_n): \quad (P\in O_S[X]_{\leq n}, Q\in K[X]_{\leq n+1},P=\partial(Q))
\implies(Q \in O_S[X]_{\leq n+1}). 
$$
Let us prove $(S_0)$. Let $P,Q$ be polynomials satisfying the premise.  
Then there exist $a_0\in O_S$ and $b_0,b_1\in K$ such that $P=a_0$ and $Q=b_0+b_1X$. 
The coefficients of $X$ and of $1$ in $P=\partial(Q)$ respectively yields the equations  
$$
\partial(b_1)=0, \quad a_0+b_1\cdot E_2=\partial(b_0). 
$$
The first equation implies $b_1\in\mathbb C \subset O_S$. The relation $b_1\in\mathbb C$ then implies that the left-hand side of the 
second equation belongs to $O_S$, while $b_0\in K$. By Lem. \ref{LEMAB}, this implies $b_0\in O_S$. 
This proves that $Q \in O_S[X]_{\leq 1}$, and therefore $(S_0)$.

Assume now $(S_{n-1})$ to be true for fixed $n\geq 1$, and let us prove $(S_n)$. Let 
$P,Q$ be polynomials satisfying the premise. Let $a_0,\ldots,a_n\in O_S$ and  
$b_0,\ldots,b_{n+1}\in K$ be such that $P=a_0+\cdots+a_nX^n$ and $Q=b_0+\cdots+b_{n+1}X^{n+1}$. 
The coefficients of $X^{n+1}$ and $X^n$ in $P=\partial(Q)$ respectively yield
$$
\partial(b_{n+1})=0,\quad a_n+(n+1)b_{n+1}\cdot E_2=\partial(b_n). 
$$
The first equation implies $b_{n+1}\in\mathbb C\subset O_S$. 
The relation $b_{n+1}\in\mathbb C$ then implies that the left-hand side of the 
second equation belongs to $O_S$, while $b_n\in K$. By Lem. \ref{LEMAB}, this implies $b_n\in O_S$. 
Let us set $\tilde Q:=Q-b_{n+1}X^{n+1}-b_nX^n$ and $\tilde P:=\partial(\tilde Q)$. Then 
$\tilde Q\in K[X]_{\leq n-1}$, which implies $\tilde P\in K[X]_{\leq n-1}$ (because $\tilde P=\partial(\tilde Q)$ and 
$\partial$ maps $K[X]_{\leq n-1}$ to itself) and $\tilde P\in O_S[X]$ (because $\tilde P=P-\partial(b_{n+1}X^{n+1}+b_nX^n)$), 
and therefore $\tilde P\in O_S[X]_{\leq n-1}$. The relation $\tilde Q\in K[X]_{\leq n-1}$ also implies 
$\tilde Q\in K[X]_{\leq n}$. Hence $(\tilde P,\tilde Q)$ satisfy the assumptions of $(S_{n-1})$. 
It follows that $\tilde Q\in O_S[X]_{\leq n-1}$. Since $b_n,b_{n+1}\in O_S$, then $Q=\tilde Q+b_{n+1}X^{n+1}+b_nX^n\in O_S[X]_{\leq n+1}$. Therefore $(P,Q)$ satisfy the conclusion of $(S_n)$, which proves by induction that $(S_n)$ is true for any $n\geq 0$. 

Assume now that $Q\in K[X]$ is such that $\partial(Q)\in O_S[X]$. Let $n\neq 0$ be such that 
$Q\in K[X]_{\leq n}$. Then $\partial(Q)\in K[X]_{\leq n}$ because $\partial$ takes $K[X]_{\leq n}$ to itself, 
and since $\partial(Q)\in O_S[X]$ one has $\partial(Q)\in O_S[X]_{\leq n}$. But one also has $Q\in K[X]_{\leq n+1}$, 
hence $(P,Q)$ satisfy the assumptions of $(S_n)$, and therefore $Q\in O_S[X]$.   
All this implies the inclusion $O_S[X]\cap\partial(K[X])\subset\partial(O_S[X])$. Since the opposite inclusion is obvious, 
one gets $O_S[X]\cap \partial(K[X])=\partial(O_S[X])$, which concludes the proof. 
\end{proof}

\begin{lem}\label{lem:5:23:0603}
 $O_S[X]$ is an integral domain, and $O_S[X] \hookrightarrow K(X)$ can be identified with the injection of $O_S[X]$ in its fraction field. There is a unique derivation $\partial$ of $K(X)$ extending the derivation 
 $\partial$ of $O_S[X]$. 
\end{lem}

\begin{proof}
The first statement follows from fact that $O_S$ is an integral domain. The field $K(X)$ contains $O_S[X]$ as a subalgebra, 
and coincides with the smallest of its subfields containing $O_S[X]$, which implies the identification of 
$\mathrm{Frac}(O_S[X])$ with $K(X)$. The last statement follows from the fact the a derivation of an integral domain 
admits a unique extension to its fraction field.    
\end{proof}

Recall the family $(1,(T_{\tilde s}g_n)_{s\in S,n\geq 1})$ of $O_S[X]$ (see Lem. \ref{lem:13:1004}). 

\begin{lem}\label{lem:freeness;0603}
The image of the family $(1,(T_{\tilde s}g_n)_{s\in S,n\geq 1})$ in $K(X)/\partial(K(X))$ is $\mathbb C$-linearly independent.
\end{lem}

\begin{proof}
Expansion for $X$ at infinity induces a field extension $K(X)\subset K((X^{-1}))$, where $K((X^{-1}))$ is the
field of Laurent series in the formal variable $X^{-1}$ with coefficients in $K$. On the other hand, there is a 
double ring inclusion $O_S[X]\subset K[X]\subset K(X)$, thus yielding a sequence of ring inclusions 
\begin{equation}\label{seq:morphisms}
O_S[X]\subset K[X]\subset K(X)\subset K((X^{-1})). 
\end{equation}
There is a unique derivation $\partial$ of $K[X]$, extending $\partial$ on $K$ and such that $\partial(X)=-E_2$. Since the derivation 
$\partial$ of $K(X)$ shares these properties, both derivations are compatible. The endomorphism of $K((X^{-1}))$ given by 
$\sum_{i\in\mathbb Z} f_iX^i\mapsto \sum_{i\in\mathbb Z} (\partial(f_i)-E_2(i+1)f_{i+1})X^i$ is well-defined and is a derivation of 
$K((X^{-1}))$. The restriction of this derivation to $O_S[X]$ coincides with the derivation $\partial$ of this algebra, 
therefore its restriction to $K(X)$ also coincides with the derivation $\partial$ of this field. Hence one obtains a family of 
compatible derivations on the rings in the sequence of morphisms \eqref{seq:morphisms}. From this one derives a sequence of linear maps
\begin{equation}\label{seq:lin:maps}
O_S[X]/\partial(O_S[X])\stackrel{\alpha}{\to} K[X]/\partial(K[X])\stackrel{\beta}{\to} K(X)/\partial(K(X))\stackrel{\gamma}{\to}  
K((X^{-1}))/\partial(K((X^{-1}))). 
\end{equation}
The derivation $\partial$ of $K((X^{-1}))$ is compatible with the direct sum decomposition 
$K((X^{-1}))=K[X] \oplus X^{-1}K[[X^{-1}]]$, because it is compatible with the derivation $\partial$  
of $K[X]$ and because $(\forall i \geq 0, f_i=0)\implies(\forall i \geq 0, \partial(f_i)-(i+1)E_2\,f_{i+1}=0)$. 
It follows that the map $\gamma \circ \beta$ is injective, and so that $\beta$ is injective. By Lem. \ref{MAIN POINT OF TODAY}, 
the map $\alpha$ is injective. It follows that $\beta \circ \alpha$ is injective, and in particular that it takes 
$\mathbb C$-linearly independent families to $\mathbb C$-linearly independent families. Then Lem. \ref{NEWNEWLEM}, 
which says that the image of $(1,(T_{\tilde s}g_n)_{s\in S,n\geq 1})$
in $O_S[X]/\partial(O_S[X])$ is $\mathbb C$-linearly independent, implies the statement.  
\end{proof}

The following statement is a direct consequence of the main result of \cite{DDMS}.

\begin{thm}[see \cite{DDMS}, Thm. 1]\label{thm:DDMS} 
Let $(\mathcal A, \mathbf d)$  be a commutative associative differential algebra with unit over a field $\mathbf k$ 
of characteristic $0$, and let $\mathcal C$ be a differential subfield of $\mathcal A$ (i.e. $\mathbf d(\mathcal C) \subset \mathcal C$).
Let $X$ be a set, $x\mapsto u_x$ be a map $X\to\mathcal C$, and let 
$(r,(x_1,\ldots,x_r))\mapsto f_{x_1,\ldots,x_r}$ be a map 
$\sqcup_{r \geq 0}X^r\to\mathcal A$, 
such that $f_\emptyset=1$ and $\mathbf d(f_{x_1,\ldots,x_r})=u_{x_1}\cdot f_{x_2,\ldots,x_r}$ 
for any $r \geq 1$ and $x_1,\ldots,x_r \in X$. 

If the image of the family $x\mapsto u_x$ in $\mathcal C/\mathbf d(\mathcal C)$ is 
$\mathbf k$-linearly independent, then the family $(r,(x_1,\ldots,x_r))\mapsto f_{x_1,\ldots,x_r}$ is 
$\mathcal C$-linearly independent.  
\end{thm}

\begin{proof}
The hypothesis of the statement is (iii) of Thm. 1 in \cite{DDMS}, and its conclusion is (i) in the same theorem; 
the statement then follows from Thm. 1 in {\it loc. cit}. 
\end{proof}

\begin{prop}\label{thm:main:bis}
(a) $\mathcal G$ is a free $O_S[g_1]$-module with basis  
\begin{equation}\label{basis:mathcalG}
(\GLarg{n_1 &n_2 &\ldots &n_r}{a_1 &a_2 &\ldots &a_r}{-})_{
((n_1,a_1),\ldots,(n_r,a_r)) \in \sqcup_{r \geq 0} \{(n,a) \in \mathbb Z_{\geq0} \times \tilde S|a=0\text{ if }n=0\}^r}.   
\end{equation}

(b) $\mathcal G$ is a free $O_S$-module with basis
\begin{equation}\label{POWER:TILDEGAMMA}
(g_1^i\cdot 
\GLarg{n_1 &n_2 &\ldots &n_r}{a_1 &a_2 &\ldots &a_r}{-})_{i\geq 0,
((n_1,a_1),\ldots,(n_r,a_r))\in 
\sqcup_{r \geq 0} \{(n,a) \in \mathbb Z_{\geq0} \times \tilde S|a=0\text{ if }n=0\}^r  
 }. 
\end{equation}
\end{prop}

\begin{proof}
(a) It follows from Lem. \ref{lem:5:14:0603} that \eqref{basis:mathcalG} is a generating family of $\mathcal G$ 
as an $O_S[g_1]$-module. 

Recall from Lem. \ref{lem:5:23:0603} that $O_S[g_1]$ is a integral domain; it therefore fits in an algebra inclusion 
\begin{equation}\label{ALGINC}
    O_S[g_1]\subset \mathrm{Frac}(O_S[g_1]). 
\end{equation}
On the other hand, there is an algebra inclusion
\begin{equation}\label{MODINC}
    \mathcal G\subset \mathcal O_{mer}(\tilde{\mathcal E}_S),  
\end{equation}
where $\mathcal O_{mer}(\tilde{\mathcal E}_S)$ is the algebra of all meromorphic functions on $\tilde{\mathcal E}_S$. Moreover, 
this inclusion is compatible with \eqref{ALGINC} and with the module structure of both sides of \eqref{MODINC} over 
the corresponding  algebras of \eqref{ALGINC}. 

Let us show that the image of \eqref{basis:mathcalG} under \eqref{MODINC} is $\mathrm{Frac}(O_S[g_1])$-linearly independent. 
Set $(\mathcal A,\mathbf d):=(\mathcal O_{mer}(\tilde{\mathcal E}_S),\partial)$, set $\mathbf k:=\mathbb C$, 
set $\mathcal C:=\mathrm{Frac}(O_S[g_1])$, set $X:=\{(n,a) \in \mathbb Z_{\geq0} \times \tilde S\,|\,a=0\text{ if }n=0\}$, 
let $x\mapsto u_x$ be the map $X\to\mathcal C$ given by $(0,0)\mapsto 1$ and 
$(n,a)\mapsto T_ag_n$ for $n>0$, let $(r,(x_1,\ldots,x_r))\mapsto f_{x_1,\ldots,x_r}$ be the map 
$\sqcup_{r \geq 0}X^r\to\mathcal A$ given by \eqref{basis:mathcalG}. 

It follows from \eqref{DE:1104} that the identities $f_\emptyset=1$ and $\mathbf d(f_{x_1,\ldots,x_r})=u_{x_1}\cdot f_{x_2,\ldots,x_r}$ 
are satisfied. It follows from Lem. \ref{lem:freeness;0603} and from the identification $\mathrm{Frac}(O_S[g_1])\simeq K(X)$ (see Lem. 
\ref{lem:5:23:0603}) that the image of $x\mapsto u_x$ in $\mathcal C/\mathbf d(\mathcal C)$ is $\mathbb C$-linearly independent. 
We can therefore apply Thm.~\ref{thm:DDMS}, which implies that the image of 
\eqref{basis:mathcalG} under \eqref{MODINC} is $\mathrm{Frac}(O_S[g_1])$-linearly independent.
By the injectivity of the maps \eqref{ALGINC} and \eqref{MODINC}, this implies the announced statement. 

(b) Let $\mathcal F:=\sqcup_{r \geq 0} \{(n,a) \in \mathbb Z_{\geq0} \times \tilde S\,|\,a=0\text{ if }n=0\}$. 
Let $\mathbb C^{(\mathcal F)}$ be the set of finitely supported complex functions on $\mathcal F$.  
Then (a) says that the map  $O_S[g_1] \otimes \mathbb C^{(\mathcal F)}\to \mathcal G$ induced by the family from (a) 
is a linear isomorphism. Moreover, Lem. \ref{lem:iso} implies that the map $O_S\otimes\mathbb C^{(\mathbb Z_{\geq0} \times \mathcal F)}
\to O_S[g_1] \otimes \mathbb C^{(\mathcal F)}$ given by $f \otimes \delta_{(i,x)} \mapsto f\cdot g_1^i \otimes \delta_x$  
for any $x\in \mathcal F$, $i\in\mathbb Z_{\geq 0}$, $f\in O_S$, is a linear isomorphism. Composing
these two linear isomorphisms, one obtains that the map $O_S\otimes\mathbb C^{(\mathbb Z_{\geq0} \times \mathcal F)}
\to\mathcal G$ induced by the family from (b) is a linear isomorphism, which proves (b). 
\end{proof}

\begin{thm}\label{THM:MAIN}
Both families \eqref{basis:HL:g=1} and \eqref{POWER:TILDEGAMMA}       
are $O_S$-bases of the $O_S$-module $A_{\mathcal E_S}$. 
\end{thm} 


\begin{proof}
The statement on the family \eqref{basis:HL:g=1} follows from Lem. \ref{cor:2203:1124}. 
The statement on the family \eqref{POWER:TILDEGAMMA}
follows from Prop. \ref{thm:main:bis}(b) and from Thm. \ref{thm:main:2606}.     
\end{proof}

\subsection{Examples of relations between elliptic HLs and the functions $\tilde\Gamma$}\label{TODO}

The fact that the iterated integration basepoint for the elliptic HLs $L_{\alpha_{i_1}\cdots \alpha_{i_n}}$ from \S\ref{NEWSECTION} is a fixed point $z_0$ \emph{different} from~$0$, whereas the iterated integration basepoint for the functions $\tilde\Gamma$ is precisely the point~$0$, makes it a bit cumbersome to pass from one basis to the other. We provide below a few examples of expressions which relate elements of the two basis in the simplest setting where $S=pr(0)$ consists only of one point. In this case, the basis of differential forms considered in \S\ref{NEWSECTION} reduces to $\alpha:=dz$ and $\beta:=E_2\,dz$. 

Let us introduce the lighter notation
$$
\tilde\Gamma(n_1 \ldots n_r;-):=\GLarg{n_1 &n_2 &\ldots &n_r}{0 &0 &\ldots &0}{-}.
$$
Since $\Gamma(\underbrace{0,\ldots ,0}_n;z)=\frac{z^n}{n!}$ and $L_{\small\underbrace{\alpha\cdots \alpha}_n}(z)=\frac{(z-z_0)^n}{n!}$, one immediately finds the formula
\begin{equation*}
L_{\small\underbrace{\alpha\cdots \alpha}_n}\,=\,\sum_{j=0}^n\frac{(-z_0)^{n-j}}{(n-j)!}\tilde\Gamma(\underbrace{0,\ldots ,0}_j;-)\,.
\end{equation*}

Furthermore, since $E_2=-g'_1$, it follows that
\begin{equation}\label{eq:240309n1}
L_{\beta}\,=\,-g_1+g_1(z_0).
\end{equation}

Combining this relation with the fact that $\tilde\Gamma(1;z)=\tilde\Gamma(1;z_0)+\int_{z_0}^zg_1(u)du$, one gets
\begin{equation*}
\tilde\Gamma(1;-)\,=\,-L_{\beta\alpha}+g_1(z_0)L_\alpha+\tilde\Gamma(1;z_0).
\end{equation*}

Moreover, from \eqref{eq:240309n1} and the shuffle product identity 
$L_\beta^k=k!L_{\scriptsize\underbrace{\beta\cdots \beta}_k}$, one gets for $n\geq 1$
\begin{align}\label{eq:240321n1}
L_{\small\underbrace{\beta\cdots \beta}_n}&=\frac{(g_1(z_0)-g_1)^n}{n!},\notag\\
g_1^n&=\sum_{k=0}^n\frac{(-1)^k\,n!}{(n-k)!}\,g_1^{n-k}(z_0)L_{\small\underbrace{\beta\cdots \beta}_k}.
\end{align}

Finally, let us express $\tilde\Gamma(2;-)$ in terms of elliptic HLs. By eq. \eqref{defgn}, one has $g_2=(g_1^2-E_2+e_2)/2$. Moreover, eq. \eqref{eq:240321n1} for $n=2$ gives $g_1^2=2L_{\beta\beta}-2g_1(z_0)L_{\beta}+g_1^2(z_0)$. Combining these two observations, one gets
\begin{align*}
\tilde\Gamma(2;-)=L_{\beta\beta\alpha}-g_1(z_0)L_{\beta\alpha}+\frac{e_2+g_1^2(z_0)}{2}\,L_{\alpha}-\frac{1}{2}\,L_{\beta}+\tilde\Gamma(2;z_0).
\end{align*}

\section{Relation of the functions $\tilde\Gamma$ with the functions $\mathrm{E_3}$}\label{sect:6:0407}

 The purpose of this section is to give an alternative proof, based on the results of \S\ref{sect:5:2606}, 
 of the fact that an algebra $\mathcal A_3$ attached to an elliptic curve 
 $\mathcal E_{alg}$ equipped with a degree 2 ramified covering $\mathcal E_{alg}\to\mathbb P^1$ which was constructed in \cite{BDDT} is stable 
 under integration. In \S\ref{sect:uniformization:2704}, we relate the framework of the present paper with the one of \cite{BDDT}. 
 In \S\ref{sect:6:2:2606}, we attach to each finite subset $S_0\subset\mathbb P^1_{\mathbb C}$ containing $\infty$ an algebra 
 $\mathcal A_3(S_0)$, which we prove to be stable under integration using Thm. \ref{thm:main:2606}; we derive the fact that 
 the inductive limit of the algebras $\mathcal A_3(S_0)$ is stable under integration, and identify this 
 inductive limit with the algebra $\mathcal A_3$ of \cite{BDDT}, thereby giving an alternative proof of the 
 result of this paper (Prop. \ref{prop:bddt:2704}(c)).

\subsection{Reminders on uniformisation of elliptic curves}\label{sect:uniformization:2704}

Let us denote by $(z,\tau)\mapsto E_2(z|\tau)$ (resp. $\tau\mapsto e_2(\tau)$) the meromorphic (resp. holomorphic) function 
on $\mathbb C\times\mathfrak H$ (resp. $\mathfrak H$) such that for any 
$\tau\in\mathfrak H$, the map $z\mapsto E_2(z|\tau)$ (resp. the element $e_2(\tau)\in\mathbb C$) coincides with the function~$E_2$ (resp. the number $e_2\in\mathbb C$) from \S\ref{sect:3:1:2704}. For $\tau \in\mathfrak H$, we set $\wp(z|\tau):=E_2(z|\tau)-e_2(\tau)$. 

Denote by $\mathbb C^3_*$ be the complement of the diagonals in $\mathbb C^3$. 

\begin{lem}\label{lem:unif1:2704}
Let $(\mathbf a_1,\mathbf a_2,\mathbf a_3)\in\mathbb C^3_*$. There exists $\tau\in\mathfrak H$, $a\in\mathbb C^\times$ and 
$b\in\mathbb C$ such that 
$$
(\mathbf a_1,\mathbf a_2,\mathbf a_3)=(a\wp(1/2|\tau)+b,a\wp(\tau/2|\tau)+b,a\wp((\tau+1)/2|\tau)+b).
$$    
\end{lem}

\begin{proof}
The set $\mathbb C^3_*$ is equipped with the commuting actions of the 
group $\mathbb C^\times\ltimes\mathbb C$ by $(a,b)\cdot (\mathbf a_1,\mathbf a_2,\mathbf a_3):=(a\mathbf a_1+b,
a\mathbf a_2+b,a\mathbf a_3+b)$ and of the symmetric group 
$\mathfrak S_3$ by permutation. The set $\mathbb C\smallsetminus\{0,1\}$ is equipped with an action of $\mathfrak S_3$, generated by the
involutions $\lambda\mapsto 1-\lambda$ and $\lambda\mapsto 1/\lambda$. 

The map $\mathbb C^3_*\to\mathbb C\smallsetminus\{0,1\}$ given by $(\mathbf a_1,\mathbf a_2,\mathbf a_3)\mapsto 
\mathbf a_{21}/\mathbf a_{31}$, where we set 
$\mathbf a_{ij}=\mathbf a_i-\mathbf a_j$, is $\mathfrak S_3$-equivariant; 
it can be identified with the projection $\mathbb C^3_*\to(\mathbb C^\times\ltimes\mathbb C)\backslash\mathbb C^3_*$ 
of $\mathbb C^3_*$ on its set of orbits under the action of $\mathbb C^\times\ltimes\mathbb C$. 

The map $\mathbb C\smallsetminus\{0,1\}\to\mathbb C$ induced by $\lambda\mapsto 256(1-\lambda+\lambda^2)^3/(\lambda^2
(1-\lambda)^2)$ is $S_3$-invariant, and can be identified with the projection $\mathbb C\smallsetminus\{0,1\}\to 
\mathfrak S_3\backslash(\mathbb C\smallsetminus\{0,1\})$ of $\mathbb C\smallsetminus\{0,1\}$
on the set of its orbits under the action of $\mathfrak S_3$.  

The map $j : \mathfrak H\to\mathbb C$ defines a bijection 
$j : \mathrm{SL}_2(\mathbb Z)\backslash\mathfrak H\to\mathbb C$. On the other hand, the map 
$\lambda : \mathfrak H\to\mathbb C\smallsetminus\{0,1\}$ defined by $\tau\mapsto (\wp(\tau/2|\tau)-\wp(1/2|\tau))/
(\wp((1+\tau)/2|\tau)-\wp(1/2|\tau))$ is compatible with the group morphism $\mathrm{SL}_2(\mathbb Z)\to
\mathrm{SL}_2(\mathbb F_2)\simeq \mathfrak S_3$; it therefore defines a map 
$\lambda : \Gamma(2)\backslash\mathfrak H\to\mathbb C\smallsetminus\{0,1\}$, which is a bijection. 
The map $\lambda$ is also the composition with the projection $\mathbb C^3_*\to\mathbb C\smallsetminus\{0,1\}$
of the map $\mathfrak H\to \mathbb C^3_*$ given by $\tau\mapsto(\wp(1/2|\tau),\wp(\tau/2|\tau),\wp((\tau+1)/2|\tau))$. 

The situation is summarized in the diagram 
$$
\xymatrix{\mathbb C^3_*\ar^{\mathbb C^\times\ltimes\mathbb C}[r]&\mathbb C\smallsetminus\{0,1\}\ar^{\mathfrak S_3}[r]&
\mathbb C\\
\mathfrak H\ar[u]\ar@/_2pc/^{\mathrm{SL}_2(\mathbb Z)}[rr]\ar^{\Gamma(2)}[r]&\Gamma(2)\backslash\mathfrak H\ar^{\mathfrak S_3}[r]\ar^{\lambda}_\sim[u]&
\mathrm{SL}_2(\mathbb Z)\backslash\mathfrak H\ar_{j}^\sim[u]}
$$
where the notation $X\stackrel{\Gamma}{\to}Y$ means that the map $X\to Y$ is $\Gamma$-invariant, and sets up 
a bijection $\Gamma\backslash X\stackrel{\sim}{\to} Y$.  

Let now $(\mathbf a_1,\mathbf a_2,\mathbf a_3)\in\mathbb C^3_*$. Its image by the composition $\mathbb C^3_*\to \mathbb C\smallsetminus\{0,1\}
\to\mathbb C\stackrel{j^{-1}}{\to}\mathrm{SL}_2(\mathbb Z)\backslash\mathfrak H$ is a well-defined $\mathrm{SL}_2(\mathbb Z)$-orbit 
in $\mathfrak H$. Let $\tau_0\in\mathfrak H$ be an element of this orbit. By the properties of the above diagram, there exists 
$\sigma\in\mathfrak S_3$ such that $\lambda(\tau_0)\in\mathbb C\smallsetminus\{0,1\}$ is related to $\mathbf a_{21}/\mathbf a_{31}$ by 
$\mathbf a_{21}/\mathbf a_{31}=\sigma\cdot \lambda(\tau_0)$. Let $\tilde\sigma\in\mathrm{SL}_2(\mathbb Z)$ be a lift of $\sigma$, one then 
has $\mathbf a_{21}/\mathbf a_{31}=\lambda(\tau)$, where $\tau=\sigma\cdot\tau_0$. Then the elements $(\mathbf a_1,\mathbf a_2,\mathbf a_3)$ and 
$(\wp(1/2|\tau),\wp(\tau/2|\tau),\wp((\tau+1)/2|\tau))$ of $\mathbb C^3_*$
have the same image in $\mathbb C\smallsetminus\{0,1\}$, which implies that they are related by the action of 
$\mathbb C^\times\ltimes\mathbb C$. 
\end{proof}

In the rest of \S\ref{sect:uniformization:2704}, we fix $\vec{\mathbf  a}=(\mathbf a_1,\mathbf a_2,\mathbf a_3)\in\mathbb C^3_*$.  

\begin{defn}\label{DEF:E:ALG}
$\mathcal E_{\mathrm{alg}}$ is the projective curve in $\mathbb P^2(\mathbb C)$ defined by 
the equation $Y^2T=(X-\mathbf a_1T)(X-\mathbf a_2T)(X-\mathbf a_3T)$, where $[X:Y:T]$ is the canonical system of projective coordinates on 
$\mathbb P^2(\mathbb C)$. 
\end{defn} 

\begin{lem}\label{lem:6:3:toto}
Let $(a,b,\tau)\in\mathbb C^\times\times\mathbb C\times\mathfrak H$ be as in Lem. 
\ref{lem:unif1:2704}, let $\Lambda:=\mathbb Z+\mathbb Z\tau$ and $\mathcal E:=\mathbb C/\Lambda$. 
Let $a^{3/2}$ be a square root of $a^3$, and denote by $(z,\tau)\mapsto\wp'(z|\tau)$ the partial 
derivative of $(z,\tau)\mapsto\wp'(z|\tau)$ with respect to $z$. 

(a) There is a unique isomorphism $iso : \mathcal E\to \mathcal E_{\mathrm{alg}}$, given by 
$pr(z)\mapsto [a\wp(z|\tau)+b:(1/2)a^{3/2}\wp'(z|\tau):1]$ for $z\notin\Lambda$ and $pr(0)\mapsto [0:1:0]$. 

(b) The image by $iso$ of $pr(1/2)$, $pr(\tau/2)$ and $pr((1+\tau)/2)$ are respectively $[\mathbf a_1:0:1]$,
$[\mathbf a_2:0:1]$ and $[\mathbf a_3:0:1]$. The isomorphism $iso$ intertwines the involution $pr(z)\mapsto pr(-z)$ of~$\mathcal E$ with the involution $[X:Y:T]\mapsto[X:-Y:T]$ of $\mathcal E_{\mathrm{alg}}$. 
\end{lem}

\begin{proof}
(a) follows from the fact that functions $(z,\tau)\mapsto \wp(z|\tau)$ and $\wp'(z|\tau)$ satisfy the identity 
$\wp'(z|\tau)^2=4(\wp(z|\tau)-\wp(1/2|\tau))(\wp(z|\tau)-\wp(\tau/2|\tau))(\wp(z|\tau)-\wp((1+\tau)/2|\tau))$ (see \cite{Kn}, Thm.~6.15) 
and from the relation from Lem. \ref{lem:unif1:2704}. The first part of (b) follows from Lem. \ref{lem:unif1:2704}, from 
the vanishing for fixed $\tau$ of the function  
$z\mapsto \wp'(z|\tau)$ at $1/2$, $\tau/2$ and $(1+\tau)/2$, which follows from its oddness and from its $\Lambda$-periodicity, and its 
second part follows from the evenness of the function $z\mapsto\wp(z|\tau)$ and oddness of the function $z\mapsto \wp'(z|\tau)$ for 
$\tau$ being fixed.    
\end{proof}

\subsection{An alternative proof of a result of \cite{BDDT}}\label{sect:6:2:2606}

Let $\vec{\mathbf a}=(\mathbf a_1,\mathbf a_2,\mathbf a_3)\in\mathbb C^3_*$, let  
$(a,b,\tau)\in\mathbb C^\times\times\mathbb C\times\mathfrak H$ be as in Lem. \ref{lem:unif1:2704} and 
let $\mathcal E,\mathcal E_{\mathrm{alg}}$ be as in Lem. \ref{lem:6:3:toto} and Def. \ref{DEF:E:ALG}. 

Let $\pi : \mathcal E_{\mathrm{alg}}\to\mathbb P^1(\mathbb C)$ be the morphism given by the composition 
$\mathcal E_{\mathrm{alg}}\subset\mathbb P^2(\mathbb C)\to\mathbb P^1(\mathbb C)$, where the last 
morphism is $[X:Y:T]\mapsto[X:T]$. 
Let $S_0 \subset\mathbb P^1(\mathbb C)$ be a finite subset containing~$\infty$. 
Let $S_{\mathrm{alg}}:=\pi^{-1}(S_0)$ and $S:=iso^{-1}(S_{\mathrm{alg}})$.  

Then $S$ is a finite subset of $\mathcal E$, which is stable under the involution $pr(z)\mapsto pr(-z)$, and 
which contains $\{pr(0)\}$; moreover, there is an isomorphism 
$$
iso : \mathcal E\smallsetminus S\to \mathcal E_{\mathrm{alg}}\smallsetminus S_{\mathrm{alg}}. 
$$

In \cite{BDDT}, (3.16), one defines the family of multivalued holomorphic functions 
$$
x\mapsto \mathrm{E_3}(\begin{smallmatrix}
n_1&\cdots&n_k\\ c_1&\cdots&c_k\end{smallmatrix};x)
$$
on $\mathcal E_{\mathrm{alg}}\smallsetminus S_{\mathrm{alg}}$ (denoted $x\mapsto \mathrm{E_3}(\begin{smallmatrix}
n_1&\cdots&n_k\\ c_1&\cdots&c_k\end{smallmatrix};x,\vec{\mathbf a})$ in {\it loc. cit.} to underscore the dependence 
in $(\mathbf a_1,\mathbf a_2,\mathbf a_3)$), indexed by tuples 
$(n_1,c_1),\ldots,(n_k,c_k)$ in $\sqcup_{k\geq 0}I^k$, where 
$I:=(\mathbb Z\times S_0^{unr})\sqcup(\mathbb Z_{\geq 0}\times S_0^{ram})$ and 
$S_0^{ram}:=S_0\smallsetminus(S_0\cap\{\mathbf a_1,\mathbf a_2,\mathbf a_3,\infty\})$
and $S_0^{unr}:=S_0\smallsetminus S_0^{ram}$; one therefore has
for each $i\in[\![1,k]\!]$ the relations $(n_i,c_i)\in \mathbb Z\times S_0$, 
with\footnote{The range of values $\mathbb Z\times\mathbb P^1_{\mathbb C}$ for the pairs $(n_i,c_i)$ 
announced in \cite{BDDT}, eq. (3.16) should in fact be restricted by the condition that $n_i\geq 0$  
whenever $c_i\in\{\mathbf a_1,\mathbf a_2,\mathbf a_3,\infty\}$, as one can derive from (3.31) 
and the discussion following it.} 
$n_i\geq 0$ whenever $c_i\in \{\mathbf a_1,\mathbf a_2,\mathbf a_3,\infty\}$.   

Recall that $\big(z\mapsto \tilde\Gamma(\begin{smallmatrix}
n_1&\cdots&n_k\\ a_1&\cdots&a_k\end{smallmatrix};z)\big)$
is a family of multivalued holomorphic functions on $\mathcal E\smallsetminus S$, indexed by tuples 
$(n_1,a_1),\ldots,(n_k,a_k)$ in $\sqcup_{k\geq 0}(\mathbb Z_{\geq 0}\times \tilde S)^k$. 

\begin{prop}[see also \cite{BDDT}, \S5] \label{prop:iso:spaces:0307}
The map $iso^*$ induces an isomorphism of vector spaces between the linear spans
$\mathrm{Span}_{\mathbb C}\{x\mapsto \mathrm{E_3}(\begin{smallmatrix}
n_1&\cdots&n_k\\ c_1&\cdots&c_k\end{smallmatrix};x)\,|\,k\geq 0,(n_i,c_i)\in (\mathbb Z\times S_0^{unr})\sqcup(\mathbb Z_{\geq 0}\times S_0^{ram})\}$ and $\mathrm{Span}_{\mathbb C}\{z\mapsto \tilde\Gamma(\begin{smallmatrix}
n_1&\cdots&n_k\\ a_1&\cdots&a_k\end{smallmatrix};z)\,|\,k\geq 0,(n_i,a_i)\in \mathbb Z_{\geq 0}\times \tilde S \}$.
\end{prop}

\begin{proof}
The family of functions $x\mapsto \mathrm{E_3}(\begin{smallmatrix}
n_1&\cdots&n_k\\ c_1&\cdots&c_k\end{smallmatrix};x)$ with variable $(k,(n_1,c_1),\ldots,(n_k,c_k))$
is defined in \cite{BDDT} as follows. One introduces in {\it loc. cit.} a family of multivalued meromorphic differentials 
$\varphi_n(c,t)dt$ on $\mathcal E_{\mathrm{alg}}$, where $(n,c) \in 
(\mathbb Z\times S_0^{unr})\sqcup(\mathbb Z_{\geq 0}\times S_0^{ram})$ and $S_0^{ram}:=S_0\smallsetminus(S_0\cap\{\mathbf a_1,\mathbf a_2,\mathbf a_3,\infty\})$  with sets of poles contained in 
$pr_{\mathrm{alg}}^{-1}(S_{\mathrm{alg}})$, where $pr_{\mathrm{alg}} : \tilde{\mathcal E}_{\mathrm{alg}}\to\mathcal E_{\mathrm{alg}}$
is a universal cover. 
The functions $x\mapsto \mathrm{E_3}(\begin{smallmatrix}
n_1&\cdots&n_k\\ c_1&\cdots&c_k\end{smallmatrix};x)$ are the iterated integrals of these 
differentials starting from the point $[0:\sqrt{-\mathbf a_1\mathbf a_2\mathbf a_3}:1]$ in $\mathcal E_{\mathrm{alg}}$, 
where $\sqrt{-\mathbf a_1\mathbf a_2\mathbf a_3}$ 
is a square root of $-\mathbf a_1\mathbf a_2\mathbf a_3$. The isomorphism $iso : \mathcal E\to\mathcal E_{\mathrm{alg}}$ induces an 
isomorphism
$\widetilde{iso} : \tilde{\mathcal E}\to\tilde{\mathcal E}_{\mathrm{alg}}$, and the image by 
$\widetilde{iso}^*$ of $\mathrm{Span}_{\mathbb C}(\varphi_n(c,t)dt\,|\,(n,c) \in 
(\mathbb Z\times S_0^{unr})\sqcup(\mathbb Z_{\geq 0}\times S_0^{ram}))$ is a space of multivalued meromorphic 
differentials on $\mathcal E$, with sets of poles contained in $pr^{-1}(S)$, which is shown in \cite{BDDT}, \S5 to coincide with 
$\mathrm{Span}_{\mathbb C}(dz,T_{a}g_n\cdot dz,a\in pr^{-1}(S))$. It follows from \eqref{ellpropgn} that the latter space is equal to  
$\mathrm{Span}_{\mathbb C}(dz,T_{a}g_n\cdot dz,a\in \tilde S)$. 
So the linear span of the functions 
$x\mapsto \mathrm{E_3}(\begin{smallmatrix}
n_1&\cdots&n_k\\ c_1&\cdots&c_k\end{smallmatrix};x)$
corresponds via $iso^*$ to the linear span of the iterated integrals starting at $iso(0)$ of the elements of 
$\mathrm{Span}_{\mathbb C}(dz,T_{a}g_n\cdot dz,a\in\tilde S)$. By Lem. \ref{NEWLEMMA:2704}(b), this linear span coincides with 
that of the regularized iterated integrals starting at $pr^{-1}(0)$ of the elements of 
$\mathrm{Span}_{\mathbb C}(dz,T_{a}g_n\cdot dz,a\in\tilde S)$, which is the linear span of 
$z\mapsto \tilde\Gamma(\begin{smallmatrix}
n_1&\cdots&n_k\\ a_1&\cdots&a_k\end{smallmatrix};z)$ for variable $(k,(n_1,a_1),\ldots,(n_k,a_k))$.      
\end{proof}

The authors of \cite{BDDT} define in eq. (3.23) a multivalued function $\mathrm Z_3$ on $\mathcal E_{\mathrm{alg}}\smallsetminus\{\infty\}$, closely related to a classical elliptic integral, and prove that $iso^*(\mathrm Z_3)=4g_1$ (see \cite{BDDT}, eq.~(5.11)).

\begin{prop} \label{prop:1713:2704}
The subalgebra 
$$
\mathcal A_3(S_0):=\mathcal O(\mathcal E_{\mathrm{alg}}\smallsetminus S_{\mathrm{alg}})[\mathrm{Z_3}][
\mathrm{E_3}(\begin{smallmatrix}
n_1&\cdots&n_k\\ c_1&\cdots&c_k\end{smallmatrix};-)|
k,(n_1,c_1),\ldots,(n_k,c_k)]
$$
of the algebra of holomorphic multivalued functions on $\mathcal E_{\mathrm{alg}}\smallsetminus S_{\mathrm{alg}}$
is stable in the sense of Def.~\ref{def:stable:subalg}. 
\end{prop}

\begin{proof}
Since $iso^*(\mathrm Z_3)=4g_1$, it follows that the image by $iso^*$ of this algebra is 
equal to $\mathcal G$ (see Def. \ref{defn:mathcalG:2704}). The result then follows from  Prop. \ref{lem:17:1704}. 
\end{proof}

Let us emphasize the dependence of $S_{\mathrm{alg}}$ in the finite subset $S_0$ such that $\{\infty\}\subset 
S_0\subset\mathbb P^1_{\mathbb C}$ by denoting it 
$S_{\mathrm{alg}}(S_0)$. Fix: (a) for each finite subset subset $S_0$ such that $\{\infty\}\subset 
S_0\subset\mathbb P^1_{\mathbb C}$, a universal cover 
$\widetilde{\mathcal E_{\mathrm{alg}}\smallsetminus S_{\mathrm{alg}}(S_0)}\to
\mathcal E_{\mathrm{alg}}\smallsetminus S_{\mathrm{alg}}(S_0)$; 
(b) for each pair of finite sets $S_0,S'_0$ with 
$\{\infty\}\subset S_0\subset S'_0\subset\mathbb P^1_{\mathbb C}$, a holomorphic map 
$p_{S_0,S'_0} : \widetilde{\mathcal E_{\mathrm{alg}}\smallsetminus S_{\mathrm{alg}}(S'_0)}
\to\widetilde{\mathcal E_{\mathrm{alg}}\smallsetminus S_{\mathrm{alg}}(S_0)}$
such that 
$$\xymatrix{\widetilde{\mathcal E_{\mathrm{alg}}\smallsetminus S_{\mathrm{alg}}(S'_0)}\ar^{p_{S_0,S'_0}}[r]\ar[d]&
\widetilde{\mathcal E_{\mathrm{alg}}\smallsetminus S_{\mathrm{alg}}(S_0)}\ar[d]\\
\mathcal E_{\mathrm{alg}}\smallsetminus S_{\mathrm{alg}}(S'_0)\ar[r]&
\mathcal E_{\mathrm{alg}}\smallsetminus S_{\mathrm{alg}}(S_0)}$$
commutes, such that for any triple of finite sets $S_0,S'_0,S''_0$ with 
$\{\infty\}\subset S_0\subset S'_0\subset S''_0\subset\mathbb P^1_{\mathbb C}$, 
one has $p_{S_0,S''_0}=p_{S_0,S'_0}\circ p_{S'_0,S''_0}$. One can then form the inductive limit algebra
$\varinjlim_{S_0}\mathcal O_{hol}(\widetilde{\mathcal E_{\mathrm{alg}}\smallsetminus S_{\mathrm{alg}}(S_0)})$; it contains
the algebra 
$\varinjlim_{S_0}\mathcal O(\mathcal E_{\mathrm{alg}}\smallsetminus S_{\mathrm{alg}}(S_0))$, which is the field 
$\mathbb C(\mathcal E_{\mathrm{alg}})$ of rational functions on $\mathcal E_{\mathrm{alg}}$.

Recall that the space of regular differentials on $\mathcal E_{\mathrm{alg}}$ is linearly spanned by 
$(TdX-XdT)/(YT)$. 

\begin{prop}\label{prop:bddt:2704}
(a) For each pair of finite sets $S_0,S'_0$ with 
$\{\infty\}\subset S_0\subset S'_0\subset\mathbb P^1_{\mathbb C}$, the algebra inclusion 
$p_{S_0,S'_0}^* : \mathcal O_{hol}(\widetilde{\mathcal E_{\mathrm{alg}}\smallsetminus S_{\mathrm{alg}}(S_0)})
\to\mathcal O_{hol}(\widetilde{\mathcal E_{\mathrm{alg}}\smallsetminus S_{\mathrm{alg}}(S'_0)})$ is such that 
$p_{S_0,S'_0}^* (\mathcal A_3(S_0))\subset \mathcal A_3(S'_0)$. Set $\mathcal A_3:=\varinjlim_{S_0}\mathcal A_3(S_0)$. 

(b) Let denote the endomorphism $\mathrm{int}_{(TdX-XdT)/(YT)}$ of 
$\mathcal O_{hol}(\widetilde{\mathcal E_{\mathrm{alg}}\smallsetminus S_{\mathrm{alg}}(S_0)})$ (see \S\ref{sect:21:2704}) 
as $\mathrm{int}_{S_0}$. Then for each  pair of finite sets $S_0,S'_0$ with 
$\{\infty\}\subset S_0\subset S'_0\subset\mathbb P^1_{\mathbb C}$, one has 
$\mathrm{int}_{S_0}\circ p_{S_0,S'_0}^*=p_{S_0,S'_0}^*\circ\mathrm{int}_{S'_0}$; let $\varinjlim_{S_0}\mathrm{int}_{S_0}$
be the corresponding endomorphism of $\varinjlim_{S_0}\mathcal O_{hol}(\widetilde{\mathcal E_{\mathrm{alg}}
\smallsetminus S_{\mathrm{alg}}(S_0)})$. 

(c) $\mathcal A_3$ is stable under $\varinjlim_{S_0}\mathrm{int}_{S_0}$. 
\end{prop}

\begin{proof}
(a) and (b) are straightforward, and (c) follows from Prop. \ref{prop:1713:2704}. 
\end{proof}

Note that Prop. \ref{prop:bddt:2704}(c) recovers the result from \cite{BDDT}, stated in \S3 and proved in \S6.

\subsection*{Acknowledgements}

The research of B.E. has been partially funded by ANR grant ``Project
HighAGT ANR20-CE40-0016''. The research of F.Z. has been funded by the Royal Society, under the grant
URF\textbackslash R1\textbackslash201473.


\begin{thebibliography}{BMMS}

\bibitem[Bl]{Bl} S. Bloch. Higher regulators, algebraic K-theory, and zeta functions of elliptic curves,
volume 11 of CRM Monograph Series. American Mathematical Society, Providence, RI, 2000.

\bibitem[BB]{surv:mpl} D. Bowman, D. Bradley, Multiple polylogarithms: a brief survey, 
Contemp. Math., 291, American Mathematical Society, Providence, RI, 2001, 71-92.

\bibitem[BDDT]{BDDT} J. Broedel, C. Duhr, F. Dulat, L. Tancredi, Elliptic polylogarithms and iterated integrals on elliptic
curves. Part I: general formalism. J. High Energy Phys.(2018), no. 5, 093. 


\bibitem[BK]{BK} J. Broedel, A. Kaderli, Amplitude recursions with an extra marked point,  
Commun. Number Theory Phys. 16 (2022), no. 1, 75–158.

\bibitem[BMMS]{BMMS}  J. Broedel, C.R. Mafra, N. Matthes, O. Schlotterer, Elliptic multiple zeta values and
one-loop superstring amplitudes, J. High Energy Phys. (2015), no.7, 112.

\bibitem[Br]{Br} F. Brown, Multiple zeta values and periods of moduli spaces $\overline{\mathfrak M}_{0,n}$, 
 Ann. Sci. Éc. Norm. Supér. (4) 42 (2009), no. 3, 371-489.

\bibitem[BL]{BL} F. Brown and A. Levin, Multiple elliptic polylogarithms, preprint arXiv:1110.6917.

\bibitem[BGF]{BGF} J. Burgos Gil, J. Fresan, Multiple zeta values: from numbers to motives, Clay Mathematics Proceedings,
to appear.

\bibitem[CEE]{CEE} D. Calaque, B. Enriquez, P. Etingof, 
Universal KZB equations: the elliptic case, Progr. Math., 269
Birkhäuser Boston, Ltd., Boston, MA, 2009, 165–266.

\bibitem[DHS]{DHS} E. D'Hoker, M. Hidding, O. Schlotterer, Constructing polylogarithms on higher-genus Riemann surfaces, 
preprint arXiv:2306.08644. 

\bibitem[De]{De} P. Deligne, Le groupe fondamental de la droite projective moins trois points, Math. Sci. Res. Inst. Publ. 
16, Springer-Verlag, 1989, 79-297. 

\bibitem[DDMS]{DDMS} M. Deneufchâtel, G.H.E. Duchamp, V.H.N. Minh, A.I. Solomon,
Independence of hyperlogarithms over function fields via algebraic combinatorics, pp. 127–139 in
Algebraic informatics (Linz, Austria, 2011), edited by F. Winkler, Lecture Notes in Comput. Sci.
6742, Springer, 2011.

\bibitem[EZ]{EZ2}  B. Enriquez, F. Zerbini, Analogues of hyperlogarithm functions on affine complex curves, preprint arXiv:2212.03119.  

\bibitem[Kn]{Kn} A. Knapp, Elliptic curves, Math. Notes, 40, 
Princeton University Press, Princeton, NJ, 1992,

\bibitem[LD]{LD} J.A. Lappo-Danilevsky, Mémoires sur la théorie des systèmes des équations différentielles linéaires. Chelsea 
Publishing Co., New York, N.Y., 1953. 

\bibitem[Le]{Lev} A. Levin. Elliptic polylogarithms: an analytic theory. Compositio Math., 106(3):267–
282, 1997.

\bibitem[LR]{LR} A. Levin, G. Racinet, Towards multiple elliptic polylogarithms, preprint 
arXiv:math/0703237. 

\bibitem[Ma]{Ma}  N. Matthes, Elliptic Multiple Zeta Values, PhD Thesis, Universität Hamburg, 
Fakultät für Mathematik, Informatik und Naturwissenschaften, 2016.  

\bibitem[Pa]{Pan}  E. Panzer, Feynman integrals and hyperlogarithms, PhD Thesis, arXiv:1506.07243v1, 
Humboldt-Universität zu Berlin, Mathematisch-Naturwissenschaftliche Fakultät, 2015

\bibitem[Po]{P} H. Poincaré, Sur les groupes des équations linéaires. Acta Math. 4 (1884), no. 1, 201–-312.

\bibitem[W]{Wei} A. Weil, Elliptic functions according to Eisenstein and Kronecker, 
Classics Math., Springer-Verlag, Berlin, 1999. 

\bibitem[Z1]{Zag1} D. Zagier. The Bloch-Wigner-Ramakrishnan polylogarithm function. Math. Ann.,
286 (1990), no. 1-3, 1990, 613–624.

\bibitem[Z2]{Zag} D. Zagier, Periods of modular forms and Jacobi theta functions, 
Invent. Math. 104 (1991), no. 3, 449–465.


\end{thebibliography}
\end{document}